\documentclass[11pt]{amsart}
\usepackage[margin=1in]{geometry}
\usepackage{amsmath,amssymb,amsthm,mathtools,microtype,hyperref}
\usepackage{xcolor}

\usepackage{pdfsync}
\usepackage{comment}

\newcommand{\C}{{\mathbb C}}       
\newcommand{\R}{{\mathbb R}}       
\newcommand{\N}{{\mathbb N}}

\newcommand{\vv}{{\vspace{2mm}}}

\newcommand{\abs}[1]{ \left \lvert {#1}  \right \rvert   }

\numberwithin{equation}{section}

\newtheorem{theorem}{Theorem}[section]
\newtheorem{lemma}[theorem]{Lemma}

\theoremstyle{remark}

\title[Dimension drop]{Failure of dimension drop for harmonic measure on Ahlfors regular flat boundaries}

\author{Yingying Cai}
\address{Yingying Cai: Departament de Matemàtiques, Universitat Autònoma de Barcelona, Barcelona, Catalonia.}
\email{yingying.cai@uab.cat}

\author{Cole Jeznach}
\address{Cole Jeznach: Departament de Matemàtiques, Universitat Autònoma de Barcelona, Barcelona, Catalonia.}
\email{cole.jeznach@uab.cat}

\author{Xavier Tolsa}
\address{Xavier Tolsa: ICREA, Barcelona, Departament de Matemàtiques, Universitat Autònoma de Barcelona, and Centre de Recerca Matemàtica, Barcelona, Catalonia.}
\email{xavier.tolsa@uab.cat}

\date{\today}

\thanks{The authors were partially supported by MICIU (Spain) under the grant PID2024-160507NB-I00. C.J. was also supported by NSF-DMS-2503326.
}

\begin{document}
\begin{abstract}
We develop an $N$-adic modification of the construction of David--Jeznach--Julia for compact subsets of the real line $\R\subset\C$ for which harmonic measure for the complementary domains has no dimension drop. In particular, we show that for every $s\in(0,1/2)$, there exists an $s$-Ahlfors regular compact set $K\subset\R\subset\C$ such that $\dim\omega^\infty=s$, where $\omega^\infty$ denotes harmonic measure for $\C\setminus K$ with pole at infinity. Combined with a result of the third author, this shows that the threshold $1/2$ is sharp.
\end{abstract}

\maketitle
\setcounter{tocdepth}{1} 
\tableofcontents

\section{Introduction}

Let $\Omega\subset\mathbb R^n$ be a connected domain with compact boundary. For $x\in\Omega$, harmonic measure $\omega_\Omega^x$ is the measure representing solutions of the Dirichlet problem for the Laplace equation. More precisely, given $f\in C(\partial\Omega)$, the solution of the Dirichlet problem with boundary data $f$ is the function $u\in C^2(\Omega)\cap C(\overline
\Omega)$ satisfying
\begin{align*}
	 -\Delta u & = 0 \text{ in } \Omega \\
	 u & = f \text{ on } \partial \Omega. 
\end{align*}
In the case when $\Omega$ is unbounded and $n\geq3$ one also asks $\lim_{x\to\infty}u(x)=0$. Then the harmonic measure for $\Omega$ with pole at $x$ is the measure $\omega_\Omega^x$ supported on $\partial\Omega$ such that
\[
u(x)=\int_{\partial\Omega}f(y)\,d\omega_\Omega^x(y)
\]
for every boundary data $f\in C(\partial \Omega)$.

The study of the absolute continuity properties of harmonic measure is an old topic, dating at least as far back as the fundamental work of the Riesz brothers. There, they managed to show that in a Jordan domain $\Omega \subset \C$ with rectifiable boundary, harmonic measure and arc-length are mutually absolutely continuous \cite{RR}. 
Developments in harmonic analysis and geometric measure theory over the following century
 allowed for a rather complete understanding of analogous properties for $(n-1)$-dimensional subsets $F \subset \partial \Omega$ of domains in higher dimensions $\Omega \subset \R^{n}$ for $n \ge 3$. Indeed $(n-1)$-rectifiability of the boundary plays an essential role in qualitative absolute continuity of $\omega_{\Omega}^x$ with respect to $\mathcal{H}^{n-1}|_{\partial \Omega}$ \cite{AHMMMTV16}, and to date, even geometric smoothness properties of a domain $\Omega \subset \R^{n}$ with quantitatively $(n-1)$-dimensional boundary can be \textit{characterized} by absolute continuity properties of $\omega_\Omega^x$ \cite{AHMMT20:geom_lp}. 

What remains far less clear is the behavior of $\omega_{\Omega}^x$ when $\partial \Omega \subset \R^{n}$ has fractional dimension $s \ne n-1$. One of the central questions for such boundaries is to understand the \textit{dimension} of $\omega^x$, where the dimension is defined as 
\begin{align}\label{eqn:dim_measure}
	\dim_{\mathcal{H}} \omega^x \coloneqq \inf \left \{  \dim_{\mathcal{H}} F \; : \; \text{ $F \subset \partial \Omega$ is Borel and } \omega^x(F^c) = 0   \right \},
\end{align}
and $\dim_{\mathcal{H}} F$ is the usual Hausdorff dimension of $F$. While it is true that for any planar domain $\Omega \subset \R^2 \simeq \C$, $\dim_{\mathcal{H}} \omega_\Omega^x \le 1$ \cite{Makarov85:conformal, JW88:dim}, Wolff showed that the analogue fails in higher dimensions; there is $\Omega \subset \R^3$ for which $\dim_{\mathcal{H}} \omega_\Omega^x > 2$ \cite{Wolff95:snowflake}. On the other hand, Bourgain proved that there is a dimensional constant $\epsilon(n) > 0$ so that for any domain $\Omega \subset \R^{n}$, one has that $\dim_{\mathcal{H}} \omega_{\Omega}^x \le n - \epsilon(n)$ \cite{Bourgain87:dim}. It remains a difficult open challenge to compute the sharp value of $\epsilon(n)$ for $n \ge 3$, but explicit bounds have been obtained in, for example \cite{BG24}. 

Closely related to the result of Bourgain is the so-called ``dimension drop'' phenomenon for harmonic measure $\omega_\Omega^x$ outside of fractal sets:
\begin{align}\label{eqn:dim_drop}
	\dim_{\mathcal{H}}\omega_\Omega^x < \dim_{\mathcal{H}} \partial \Omega.
\end{align}
The dimension drop property \eqref{eqn:dim_drop} has now been explored (and proved) in various fractal settings, such as for planar domains with self-conformal structures \cite{Volberg93:cantor, Batakis96:cantor, BZ15:cantor}, or domains with uniformly non-flat boundaries \cite{Azzam20:dim}, and these results led to the prevailing notion that \eqref{eqn:dim_drop} ought to be a typical behavior for most domains $\Omega \subset \R^{n}$ for which $\partial \Omega$ is fractal, e.g., if $\partial \Omega$ is $s$-Ahlfors regular with $s \ne n-1$:
\begin{align*}
	C^{-1} r^s \le \mathcal{H}^s|_{\partial \Omega }(B(x,r)) \le C r^s, \; \; x \in \partial \Omega, \;\; 0 < r < \mathrm{diam}(\partial\Omega).
\end{align*}
However, David, the second author, and Julia~\cite{DJJ23} showed that for every $0<s<0.2497$, there exists an $s$-Ahlfors regular compact set $K\subset\mathbb{R} \subset \C$ for which \eqref{eqn:dim_drop} fails for $\omega_{\C \setminus K}^x$; in fact, the harmonic measure is comparable to $\mathcal H^s|_K$.
On the other hand, the third author~\cite{T24} proved that if $\partial \Omega \subset \R \subset \C$ is $s$-Ahlfors regular with $1/2 \le s < 1$, then necessarily $\dim_{\mathcal{H}} \omega_{\Omega}^x < s = \dim_{\mathcal{H}} \partial \Omega$. Thus, the natural question left by these results is what happens in the
intermediate range
\[
0.2497\leq s<\frac12.
\]
Our main result shows that the threshold $1/2$ is sharp:
\begin{theorem}\label{thm:main}
For every $0<s<1/2$, there exists an $s$-Ahlfors regular compact set
$K\subset\mathbb{R} \subset \C$ such that, if $\omega^\infty$ denotes the harmonic
measure of $\C \setminus K$ with pole at infinity, 
\[
    \dim_{\mathcal H}\omega^\infty=s.
\]
Moreover, $\omega^\infty$ and $\mathcal{H}^s|_K$ are equivalent, i.e., there is some constant $C \ge 1$ for which 
\begin{align*}
	C^{-1} \, \mathcal{H}^s|_K(F)\le \omega^\infty(F) \le C \mathcal{H}^s|_K(F) \text{ for all } F \subset K \text{ Borel.}
\end{align*}
\end{theorem}

Theorem \ref{thm:main} along with the work of the third author \cite{T24} then gives a complete description of the dimension drop phenomenon for harmonic measure for $s$-Ahlfors regular \textit{flat} boundaries $\partial \Omega \subset \R \subset \C$: dimension drop always happens if $1/2 \le s < 1$, but may fail if $0 < s < 1/2$. In the companion paper \cite{CJT:preprint}, we explore the non-flat case (when $\partial \Omega$ is not necessarily contained in the real line), and show instead that the threshold is at $s = 1$: for each $ 0 < s < 1$, there exists $\Omega \subset \C$ so that $\partial \Omega$ is $s$-Ahlfors regular, and so that $\dim_{\mathcal{H}} \omega_\Omega^x = s$.



\subsection{Idea of the construction}

First, it is worthwhile to give an overview of the main construction and how it differs from those sets produced in \cite{DJJ23}. Like the setting there, our compact set $K \subset \R$ is a limit of a sequence of discrete sets $K_n \subset \R$, but our construction is $N$-adic: the set $K_{n}$ is produced from $K_{n-1}$ by replacing each point $p \in K_{n-1}$ by $N$ children, $x_1(p), \dotsc, x_N(p)$, which are well-spaced near the ball $B(p, r^{n-1})$, where $r \in (0,1)$ is chosen so that 
\begin{align*}
	Nr^s =1.
\end{align*}
The relationship above is the one which guarantees that our limiting set $K$ be $s$-Ahlfors regular. 
In \cite{DJJ23} only two points are used (i.e., $N=2$), but using an $N$-adic construction gives us much more flexibility.

 The main goal is to carefully choose the positions of these children $x_1(p), \dotsc, x_N(p)$ simultaneously over all $p \in K_{n-1}$ in such a way that the harmonic measure outside the limiting set $K$ is quantitatively controlled. Since the harmonic measure with pole at infinity in $\C \setminus K$ is the equilibrium measure $\mu_K$ for $K$, we do this in practice by choosing the positions of the $x_k(p)$ so that the discrete logarithmic potential on $K_n$,
\begin{align}\label{eqn:discrete_potential}
	g_{n}(x) = N^{-n} \sum_{q \in K_n, q  \ne x} \log  \abs{x-q} ,
\end{align}
is approximately constant on $K_n$. 

Our model case is a small perturbation of a self-similar construction, where $p \in K_{n-1}$ is replaced by scaled \textit{Chebyshev nodes} centered at $p$; that is, $p$ is replaced by its $N$ children defined by
\begin{align}\label{eqn:model1}
	x_k = p + a_{n-1}(p) r^{n-1} q_k, \qquad 1 \le k \le N,
\end{align}
where the $q_k$ are the zeros of the $N$-th Chebyshev polynomial (which we call Chebyshev nodes)
\begin{align*}
	q_k = \cos \left(  \frac{   (2k-1) \pi  }{2 N }  \right), \qquad 1 \le k \le N, 
\end{align*}
and $a_{n-1}(p) \ge 1$ is a scaling factor. 
The significance of these nodes is the weak convergence of the probability measures
\begin{align*}
	N^{-1} \sum_{k=1}^N \delta_{q_k} \rightharpoonup  \mu_I \equiv  \omega_{\C \setminus [-1,1] }^\infty,
\end{align*}
where $\mu_I$ is the equilibrium measure for $I=[-1,1] \subset \C$, $d\mu(x) = c(1-x^2)^{-1/2} \; dx|_{[-1,1]}$. More concretely, these nodes are chosen so that the discrete potential
\begin{align*}
	q_k \rightarrow \sum_{1 \le j \le N, j \ne k} \log \left( \abs{q_k - q_j}  \right)
\end{align*}
is almost constant in $q_k$. Notice that in the case that $N =2$ (i.e., the setting of \cite{DJJ23}) this is automatically satisfied by any relative position of two nodes, $q_1$ and $q_2$.

When one chooses $a_{n-1}(p) \equiv 1$ in \eqref{eqn:model1}, the resulting set $K$ is self-similar, and it is known that dimension of harmonic measure drops, i.e., $\dim \omega_{\C \setminus K }^\infty < s$. The main observation of \cite{DJJ23} is that one may inductively choose the ratios $a_{n-1}(p) \simeq 1$ in such a way as to force the discrete potential \eqref{eqn:discrete_potential} to have smaller oscillations as $n$ increases, and eventually conclude that $\mu_K$, the equilibrium measure for the limiting set $K$, is equivalent to $\mathcal{H}^s|_K$. With the same philosophy but with $N \gg 2$ points, we can produce the examples given by Theorem \ref{thm:main} for all $s < 1/3$, which improves upon the range previously known. See Theorem \ref{thm:one-third}.

\vspace{0.5em}

To obtain the examples for arbitrary $s < 1/2$, choosing the scaling parameter $a_{n-1}(p)$ alone does not suffice, and additional small deformations are needed. It turns out that one can improve the errors on the oscillation of $g_n(x)$ provided that the Chebyshev nodes locations are perturbed by small polynomials; that is, if instead the children $x_k$ of $p \in K_{n-1}$ are defined by 
\begin{align}\label{eqn:model2}
	x_k = p + a_{n-1}(p) r^{n-1} Q_p(q_k), \qquad 1 \le k \le N, 
\end{align}
where now each $Q_p$ is a polynomial (with possible high degree $M$) depending on $p$ and is a small perturbation of the identity. A delicate analysis of how the oscillations of $g_{n}(x_k)$ change from $g_{n-1}(p)$ to $g_{n}(x_k)$ shows that the polynomial deformation $Q_p$ as in \eqref{eqn:model2} influences this change (up to a good, high order error) via a sum of polynomials of two types:
\begin{align*}
	\sum_{m=1}^{M-1} (LP_{m+1,p}(q_k) + R_{m,p}(q_k)), 
\end{align*}
where $Q_p(t) = t + P_{2,p}(t) + \cdots + P_{M,p}(t)$, $R_{m,p}(t)$ are remainder polynomials, and $L$ is an invertible operator 
between suitable spaces of polynomials that lowers the degree by one. The key point will be that the remainders $R_{m,p}(t)$ depend on the deformation polynomials $P_{2,p}(t), \dotsc, P_{m,p}(t)$ in a triangular way, and thus this error above can be canceled entirely, leaving only a good, high order error. For more details, see Section \ref{sec:poly_corrections}

\subsection{Notation and conventions}

All of our discrete sets $K_n$ are constructed inductively, where $K_0 = \{0\}$ each $p \in K_{n-1}$ is replaced by $N$ \textit{children} $x_1, \dotsc, x_N$ associated to the Chebyshev nodes $q_1, \dotsc, q_N$ of the form 
\begin{align*}
	x_j = p + a_{n-1}(p)r^{n-1}Q_p(q_j), \; \; 1 \le j \le N,
\end{align*}
in order to form $K_n$. In this case, $p$ is said to be the \textit{parent} of each $x_j$, and we often denote this relationship by $\hat{x}_j = p$. If $x,y \in K_n$ and $\hat{x} = \hat{y}$, then $x$ and $y$ are called \textit{siblings}. In addition, given some point $x \in K_n$, then it has a parent $\hat{x} = p \in K_{n-1}$, and thus there exists some index $j = j(x)$ with $1 \le j \le N$ called the \textit{index} of $x$ so that 
 \begin{align*}
 	x = p + a_{n-1}(p)r^{n-1}Q_{p}(q_j).
 \end{align*}
 Thus the index of $x$ indicates its relative position to its parent $p$. 
  
 It shall also be useful when discussing distinct points $x,y \in K_n$ to distinguish points by their \textit{nearest common ancestor}, which is as follows. Fix $x \in K_n$, and write $\mathcal{Y}_1(x) = \{ y \in K_n \; : \; y \ne x, \hat{x} = \hat{y} \}$ for the siblings of $x$ in $K_n$. For $\ell \ge 2$, we define $\mathcal{Y}_\ell(x)$ to be the set of points in $y \in K_n$ such that the nearest common ancestor of $y$ and $x$ is $\ell$ generations above $x$. More precisely, if we write $P^{(k)}(\cdot )$ for the parent operator $\hat{(\cdot)}$ composed with itself $k$ times, then we have that for $x \in K_n$,
 \begin{align*}
 	\mathcal{Y}_\ell(x) \coloneq \{   y \in K_n  \; : \; P^{(\ell)}(x) = P^{(\ell)}(y),  \;  P^{(\ell-1)}(x) \ne  P^{(\ell-1)}(y) \}, \; 1 \le \ell \le n
 \end{align*}
 with the understanding the $P^{(0)}(y) = y$. If $y \in \mathcal{Y}_{\ell}(x)$, then $P^{(\ell)}(y) = P^{(\ell)}(x) \in K_{n-\ell} $ is said to be the nearest common ancestor of $x$ and $y$. We have the following decomposition of $K_n \setminus \{x\}$ based on the nearest common ancestor: for each $x \in K_n$,
 \begin{align}
 	K_n  = \{x\} \cup  \left( \bigcup_{\ell = 1 }^n \mathcal{Y}_\ell(x)\right),
 \end{align}
 and the union is disjoint. It shall prove useful later that (in all of our constructions), we have that $\abs{x-y} \simeq r^{n-\ell}$ for $x \in K_n$, $y \in \mathcal{Y}_\ell(x)$, and $1 \le \ell \le n$ (see Lemma \ref{lem:ahlf_reg}).

\subsection{Organization of the paper}

The paper is organized as follows.

Section~\ref{sec:chebyshev} collects several identities and estimates related to the Chebyshev nodes $q_k$ which will be used throughout the paper.

Section~\ref{sec:basic-construction} presents the basic $N$-adic
construction for the model case \eqref{eqn:model1} and proves the main result for the range $0 < s < 1/3$; see Theorem \ref{thm:one-third}.

Section~\ref{sec:first-corrections} introduces the first polynomial
corrections and shows that introducing quadratic (resp. cubic) deformations to the Chebyshev nodes leads to the thresholds $s < 2/5$ (resp. $s < 3/7$). We include these estimates because we feel these computations help clarify how the polynomial deformations appear later in our general construction for $s < 1/2$.

Section~\ref{sec:poly_corrections} introduces the polynomial correction
operator $L$, and shows how small polynomial perturbations of the Chebyshev nodes affects the oscillation of the potential $g_n(x)$. The main result in this section is Lemma \ref{lem:higher-basic-error}.

Section~\ref{sec:higher-potential} is a stand-alone section which includes basic estimates on discrete potentials, which are needed in the main construction.

Finally, Section~\ref{sec:finite-renormalization} combines the estimates of Sections~\ref{sec:poly_corrections} and \ref{sec:higher-potential} to provide our main construction, and prove our main Theorem \ref{thm:main}; see Lemma \ref{lem:finite-order-induction} and its consequence, Theorem \ref{thm:main-general-new}.

\subsection{AI Disclosure} The authors used the LLM GPT-5.6 in the development of this work to explore mathematical ideas and develop proofs. In particular, it was the idea of the authors to approximate the equilibrium measure of a segment with point masses, and Theorem \ref{thm:one-third} was known to the authors without the aid of AI. Via a thorough interplay of prompting this LLM, the authors were able to upgrade this result to the full range of dimensions $0 < s< 1/2$. The authors have read, checked, and take full responsibility for all of the proofs in this paper. Moreover, the authors have significantly modified some of the arguments provided by the LLM.

\section{Chebyshev nodes and preliminary estimates}
\label{sec:chebyshev}

We begin by collecting basic facts about the Chebyshev nodes that
will be used throughout the construction. The $N$-th Chebyshev
polynomial of the first kind, $T_N$, is characterized by
\begin{equation}\label{eq:def-TN}
    T_N(\cos\theta)=\cos(N\theta).
\end{equation}
$T_N$ is a polynomial of degree $N$, leading coefficient $2^{N-1}$, and zeros
\[
    q_j=\cos\theta_j,
    \qquad
    \theta_j=\frac{(2j-1)\pi}{2N},
    \qquad
    1\leq j\leq N.
\]
Consequently, $T_N$ can be written as 
\[
    T_N(x)=2^{N-1}\prod_{j=1}^N(x-q_j).
\]
For later use, we define the quantities
\begin{align}\label{eqn:bj}
 b_j=\sum_{k\neq j}\log|q_j-q_k|,
\qquad
b_*=\max_{1\leq j\leq N} b_j,
\qquad
\beta_j=b_j-b_*.
\end{align}

\begin{lemma}\label{lem:cheb0}
For each $N \geq2$, we have 
\[
 \sum_{j=1}^Nq_j=0\quad \text{ and } \quad
 \frac1N\sum_{j=1}^Nq_j^2=\frac12.
\]
\end{lemma}

\begin{proof}
The fact that $\sum_{j=1}^Nq_j=0$ follows because the set of zeros is symmetric about the origin.
To show the second identity, we use that
$q_j=\cos\theta_j$ and
\[
 \cos^2\theta=\frac{1+\cos2\theta}{2},
\]
while the equally spaced angles $2\theta_j$ in $[0,2\pi)$ make the sum of $\cos2\theta_j$ vanish.
\end{proof}

An important identity satisfied by $T_N$ is the following differential equation, which follows differentiating \eqref{eq:def-TN} easily:
\begin{equation}\label{eqdiff}
 (1-x^2)T_N''(x)-xT_N'(x)+N^2T_N(x)=0.
\end{equation}
The $(N-1)$-th Chebyshev polynomial of the second kind is defined by
\[
 U_{N-1}(\cos\theta)=\frac{\sin(N\theta)}{\sin\theta}.
\]
Differentiating $T_N(\cos\theta)=\cos(N\theta)$ gives
\[
 T_N'(x)=NU_{N-1}(x).
\]
At a zero $q_j=\cos\theta_j$,
\[
 |T_N'(q_j)|=\frac{N}{\sin\theta_j}.
\]
Combining this with the product factorization yields for each $1 \le j \le N$,
\begin{equation}\label{eqTX5}
 \prod_{k\ne j}|q_j-q_k|
 =\frac{N}{2^{N-1}\sin\theta_j}.
\end{equation}

\begin{lemma}[Basic identities and estimates]\label{lem:cheb}
For $1\le j\le N$,
\begin{equation}\label{eqTX20}
 b_j=\log N-(N-1)\log2-\log\sin\theta_j,
 \qquad \beta_j=\log\frac{\sin\theta_1}{\sin\theta_j}\le0.
\end{equation}
Moreover,
\begin{equation}\label{eqTX21}
 \sum_{k\ne j}\frac1{q_j-q_k}=\frac{q_j}{2(1-q_j^2)},
\end{equation}
\begin{equation}\label{eqTX22}
 \sum_{k\ne j}\frac1{|q_j-q_k|^2}\lesssim\frac{N^2}{\sin^2\theta_j},
\end{equation}
and
\begin{equation}\label{eqTX23}
 d_N:=\min_{j\ne k}|q_j-q_k|\simeq N^{-2}, \qquad \abs{ q_j - q_{j+1}  } \gtrsim \sin(\theta_j)/ N \simeq \sin(\theta_{j+1})/N.
\end{equation}
The implicit constants are absolute.
\end{lemma}

\begin{proof}
The first identity follows immediately from \eqref{eqTX5}:
$$b_j=\sum_{k\ne j}\log|q_j-q_k|=\log N-(N-1)\log2-\log\sin\theta_j.$$
This implies that
$$\beta_j=b_j-b_* =  b_j -\max_kb_k = b_j - b_1 = \log\frac{\sin\theta_1}{\sin\theta_j}\le0.
$$ 

To deduce the identity involving the reciprocal sums, notice that,
for a polynomial $P(x)=c\prod_k(x-x_k)$ with simple zeros,
\[
 \frac{P''(x_j)}{2P'(x_j)}=\sum_{k\ne j}\frac1{x_j-x_k}.
\]
Applying this to $T_N$ and using the differential equation \eqref{eqdiff} at $q_j$ gives
\begin{equation}\label{eqTX6}
 \sum_{k\ne j}\frac1{q_j-q_k} = \frac{T_N''(q_j)}{2\,T_N'(q_j)}
 =\frac{q_j}{2(1-q_j^2)}.
\end{equation}

To prove the inequality about the squared reciprocal terms,
write $T_N(x)=(x-q_j)Q_j(x)$. Then
\[
 \frac{Q_j'(q_j)}{Q_j(q_j)}=\sum_{k\ne j}\frac1{q_j-q_k},
 \qquad
 \left(\frac{Q_j'}{Q_j}\right)'(q_j)=-\sum_{k\ne j}\frac1{(q_j-q_k)^2}.
\]
Since $T_N'(q_j)=Q_j(q_j)$, $T_N''(q_j)=2Q_j'(q_j)$, and the third derivative satisfies $T_N'''(q_j)=3Q_j''(q_j)$,
\[
 \sum_{k\ne j}\frac1{(q_j-q_k)^2} = -\frac{Q_j''(q_j)}{Q_j(q_j)} + \frac{Q_j'(q_j)^2}{Q_j(q_j)^2}
 =\left(\frac{T_N''(q_j)}{2T_N'(q_j)}\right)^2-\frac{T_N'''(q_j)}{3T_N'(q_j)}.
\]
Differentiating the Chebyshev differential equation gives
\[
 (1-x^2)T_N'''(x)-3xT_N''(x)+(N^2-1)T_N'(x)=0.
\]
At $x=q_j$ and using also \eqref{eqTX6}, this yields
\[
 \frac{T_N'''(q_j)}{T_N'(q_j)}= \frac{3q_jT_N''(q_j)}{(1-q_j^2)T_N'(q_j)}-\frac{N^2-1}{1-q_j^2}
 =\frac{3q_j^2}{(1-q_j^2)^2}-\frac{N^2-1}{1-q_j^2}.
\]
Therefore
\[
 \sum_{k\ne j}\frac1{(q_j-q_k)^2}=
 \biggl(\frac{q_j}{2(1-q_j^2)}\biggr)^2 - 
 \frac13\biggl(\frac{3q_j^2}{(1-q_j^2)^2}-\frac{N^2-1}{1-q_j^2}\biggr)=
 \frac{N^2-1}{3(1-q_j^2)}-\frac{3q_j^2}{4(1-q_j^2)^2}.
\]
In particular,
\[
 \sum_{k\ne j}\frac1{|q_j-q_k|^2}
 \lesssim\frac{N^2}{\sin^2\theta_j}.
\]

To prove that $d_N\simeq N^{-2}$, 
observe that the gaps between consecutive Chebyshev nodes are smallest at the endpoints. Hence,
\[
 \min_{j\ne k}|q_j-q_k| = |q_1-q_2|
 =\cos\frac{\pi}{2N}-\cos\frac{3\pi}{2N}
 \simeq N^{-2}.
\]
\end{proof}

\begin{lemma}[Weighted reciprocal sums]\label{lem:weighted}
For every integer $m\ge2$ and every $0\le a\le m$,
\[
 \sum_{\ell\ne j}\frac{\sin^a\theta_\ell}{|q_j-q_\ell|^m}
 \lesssim_m \frac{N^m}{\sin^{m-a}\theta_j}.
\]
The same estimate holds, up to a constant depending on $m$ and the bi-Lipschitz constant, if $q_j-q_\ell$ is replaced by $Q(q_j)-Q(q_\ell)$ for a uniformly bi-Lipschitz map $Q$ on $[-1,1]$.
\end{lemma}
\begin{proof}
Use
\[
 |q_j-q_\ell|=2\sin\frac{\theta_j+\theta_\ell}{2}
 \left|\sin\frac{\theta_j-\theta_\ell}{2}\right|.
\]
We shall use the symmetric estimate
\[
 \sin\frac{\theta_j+\theta_\ell}{2}
 \ge \frac12\max\{\sin\theta_j,\sin\theta_\ell\}.
\]
Indeed, for fixed $u\in(0,\pi)$, the minimum of
\(\sin((u+v)/2)\) over \(v\in[0,\pi]\) is attained in one of the end-points. Then it follows that
\[
 \sin\frac{u+v}{2}\ge \min\left\{\sin\frac u2,\sin\frac{u+\pi}2\right\}= 
 \min\left\{\sin\frac u2,\cos\frac u2\right\}\geq \sin\frac u2\,\cos\frac u2
 \ge\frac12\sin u,
\]
and the same argument with $u$ and $v$ interchanged gives the claim. 

Since $|\sin\frac x2|\gtrsim |x|$ for $x\in[-\pi,\pi]$, we also have
\[
 \left|\sin\frac{\theta_j-\theta_\ell}{2}\right|\gtrsim\frac{|j-\ell|}{N}.
\]
Thus
$$|q_j-q_\ell| \gtrsim \max\{\sin\theta_j,\sin\theta_\ell\}\,\frac{|j-\ell|}{N},$$
and so
\[
 \frac{\sin^a\theta_\ell}{|q_j-q_\ell|^m}
 \lesssim_m \frac{N^m}{\sin^{m-a}\theta_j}\frac1{|j-\ell|^m}.
\]
Now the result follows by summation.
\end{proof}

The preceding estimates also imply that the sets \(K_n\) have geometric properties similar to those of self-similar Cantor sets, as shown in the following lemma.

\begin{lemma}\label{lem:ahlf_reg}
	Let $ 0 < s < 1/2$, let $N \in \N$, $Nr^s = 1$, set $K_0 = \{0\}$ and define $K_n$ inductively for $n \ge 1$ by 
	\begin{align*}
		K_n = \bigcup_{p \in K_{n-1}} \{   p + a_{n-1}(p) r^{n-1} Q_p(q_j) \; : \; 1 \le j \le N \},
	\end{align*}
	where $a_0(0)=1$ and $1 \le a_{n-1}(p) \le C_1 N \sin(\theta_{j(p)})$ in the case $n\geq2$, and each $Q_p:[-1,1] \rightarrow [-2,2]$ satisfies
	\begin{align*}
		\abs{q_j - q_k} /2 \le \abs{  Q_p(q_j) - Q_p(q_k)  } \le 2 \abs{ q_j -q_k  }, \; \; 1 \le j,k \le N.
	\end{align*}
	Then as long as $N$ is large enough depending on $C_1$ and $s$, then for each $x \in K_n$ and $1 \le \ell \le n$,
	\begin{align}\label{eqn:nca}
		\# \mathcal{Y}_\ell(x) = (N-1)N^{\ell-1}, \qquad  N^{-2} r^{n-\ell}  \lesssim \abs{x-y} \lesssim C_1 N r^{n-\ell} \text{ for all } y \in \mathcal{Y}_\ell(x).
	\end{align}
	Consequently, $K_n \rightarrow K$ in the Hausdorff distance, where $K$ is a compact, $s$-Ahlfors regular set, and 
	\begin{align*}
		\mu_n \coloneqq N^{-n} \sum_{x \in K_n} \delta_{x} \rightharpoonup \mu,
	\end{align*}
	where $\mu$ is an $s$-Ahlfors regular measure supported on $K$. 
\end{lemma}

\begin{proof}
Let $x\in K_n$, and let $y\in\mathcal{Y}_\ell(x)$
for some $1\leq\ell\leq n$.

If $\ell=1$, then, writing $p=\hat{x}$, we have
\[
|x-y|
=
a_{n-1}(p)r^{n-1}
|Q_p(q_{j(x)})-Q_p(q_{j(y)})|.
\]
Thus, by the bi-Lipschitz property of $Q_p$ and
Lemma~\ref{lem:cheb},
\begin{equation}\label{eqn:sibling_dist}
cr^{n-1}\frac{\sin\theta_{j(x)}}N
\leq |x-y|
\leq 4C_1Nr^{n-1},
\end{equation}
for some absolute constant $c>0$.
This proves the desired estimates for $\ell=1$.

Suppose now that $\ell\geq2$. Set
\[
u=P^{(\ell-1)}(x),
\qquad
v=P^{(\ell-1)}(y).
\]
Then $u$ and $v$ are distinct children of their
common parent
\[
p=P^{(\ell)}(x)=P^{(\ell)}(y).
\]
Let $j=j(u)$ and $k=j(v)$, so that $j\ne k$.

We first estimate the distances from $x$ and $y$
to $u$ and $v$, respectively. By the bounds
on the scaling factors and on the polynomials
$Q_p$, we have
\[
\begin{aligned}
|x-u|
&\leq
CC_1N\left(r^{n-1}+\cdots+r^{n-\ell+2}\right)
+CC_1Nr^{n-\ell+1}\sin\theta_j
\\
&\leq
CC_1Nr^{n-\ell+1}
\left(r+\sin\theta_j\right)
\leq
CC_1Nr^{n-\ell+1}\sin\theta_j,
\end{aligned}
\]
where we used that
$r=N^{-1/s}<N^{-1}\lesssim\sin\theta_j$.

The same argument, applied to $y$, gives
\[
|y-v|
\leq
CC_1Nr^{n-\ell+1}\sin\theta_k.
\]
Consequently,
\begin{equation}\label{eqn:two-tails}
|x-u|+|y-v|
\leq
CC_1Nr^{n-\ell+1}
\max\{\sin\theta_j,\sin\theta_k\}.
\end{equation}

On the other hand, since $u$ and $v$ are
siblings, the bi-Lipschitz property of $Q_p$
and the estimate in the proof of
Lemma~\ref{lem:weighted} yield
\[
\begin{aligned}
|u-v|
&=
a_{n-\ell}(p)r^{n-\ell}
|Q_p(q_j)-Q_p(q_k)|
\geq
\frac12 r^{n-\ell}|q_j-q_k|
\\
&\geq
c r^{n-\ell}
\max\{\sin\theta_j,\sin\theta_k\}
\frac{|j-k|}{N}
\geq
c r^{n-\ell}
\frac{\max\{\sin\theta_j,\sin\theta_k\}}N.
\end{aligned}
\]

Combining this estimate with
\eqref{eqn:two-tails}, we obtain
\[
\begin{aligned}
|x-y|
&\geq |u-v|-|x-u|-|y-v|
\\
&\geq
r^{n-\ell}
\max\{\sin\theta_j,\sin\theta_k\}
\left(\frac cN-CC_1Nr\right)
\\
&\geq
\frac{c}{2N}r^{n-\ell}
\max\{\sin\theta_j,\sin\theta_k\}
\geq cr^{n-\ell}N^{-2},
\end{aligned}
\]
provided that $N$ is sufficiently large
so that $CC_1N^2r\leq c/2$.
This is possible because $s<1/2$ and
$r=N^{-1/s}$.

For the upper bound, using the same
decomposition and the preceding estimates,
we have
\[
\begin{aligned}
|x-y|
&\leq |x-u|+|u-v|+|v-y| \leq CC_1Nr^{n-\ell}.
\end{aligned}
\]
This proves the desired estimates on $|x-y|$.
The identity
\[
\#\mathcal{Y}_\ell(x)=(N-1)N^{\ell-1}
\]
follows directly from the tree structure.

Finally, the convergence of $K_n$ in the
Hausdorff distance and the $s$-Ahlfors
regularity of the limiting measure follow
by the standard symbolic-space argument,
using the separation and diameter bounds
established above. See, for example,
\cite[Proposition 2.1]{DJJ23}.
\end{proof}

\section{The basic construction: the range $s<1/3$}\label{sec:basic-construction}

Our goal in this section is to prove Theorem \ref{thm:main} for $0 < s < 1/3$ in the model case as in \eqref{eqn:model1}. Hence we assume throughout this section that $0<s<1/3$. Set
$
    K_0=\{0\}.
$
We construct the sets $K_n$ inductively. Assuming $K_{n-1}$ is constructed, then for each $p\in K_{n-1}$, we choose a scalar
$a_{n-1}(p)\geq 1$ depending just on $p$ and define the $N$ children of $p$ by
\[
    x_{k}
    =
    p+a_{n-1}(p)r^{n-1}q_k,
    \qquad 1\leq k\leq N.
\]
Let $K_n$ be the set of all such children. Next, we define
\[
 \mu_n=N^{-n}\sum_{x\in K_n}\delta_x, \qquad 
 g_n(x)=N^{-n}\sum_{\substack{y\in K_n\\y\ne x}}\log|x-y|.
\]
Notice that $g_n$ is the discrete logarithmic potential of $\mu_n$, which we claim has vanishing oscillation as $n \rightarrow \infty$. More precisely, our main result in this section is the following.

\begin{lemma}[Key induction lemma, $s < 1/3$]
\label{lem:basic-ind} Fix $\eta > 0$ and $0 < s < 1/3$. For $N$ sufficiently large depending on $s$ and $\eta$, there exists a choice of $a_{n-1}(p) \ge 1$ so that for each
$n\geq1$ and $x\in K_n$,
\[
    g_n(x)
    =
    c_n
    +
    N^{-n}
    \bigl(
        \beta_{j(x)}+e_n(x)
    \bigr),
    \qquad
    |e_n(x)|\leq\eta.
\]
Moreover, for every $n\geq2$ and every $p\in K_{n-1}$,
\[
    1
    \leq
    a_{n-1}(p)
    \lesssim_\eta
    N\sin\theta_{j(p)},
\]
and $c_n \in \R$ converges as $n \rightarrow \infty$. 
\end{lemma}

The way we prove Lemma \ref{lem:basic-ind} is through an inductive argument, choosing the $a_{n-1}(p)$ one generation at a time, based on the following decomposition of $g_{n}(x) - g_{n-1}(\hat{x})$. Notice that since $ -C_N \le \beta_{j(x)} \le 0$ is bounded by a constant depending on $N$, the oscillation of $g_n$ about the constants $c_n$ vanishes as $n \rightarrow \infty$. 

\subsection{The DJJ decomposition}
Let $x\in K_n$ and $p=\widehat x$ be the parent of $x$. As in \cite{DJJ23}, our goal is to estimate the difference
\[
 g_n(x)-g_{n-1}(p)=N^{-n}(\Delta_1+\Delta_2+\Delta_3),
\]
where we split the difference in three terms,
\begin{equation}\label{eqn:delta}
	\begin{split}
	 \Delta_1& =\sum_{\substack{y\in K_n\\\widehat y=p,\ y\ne x}}\log|x-y|, \\
	 \Delta_2& =N\sum_{\substack{y\in K_{n-1}\\y\ne p}}
	 \bigl(\log|x-y|-\log|p-y|\bigr), \\
	 \Delta_3& =\sum_{\substack{y\in K_{n-1}\\y\ne p}}\sum_{\widehat z=y}
	 \bigl(\log|x-z|-\log|x-y|\bigr).
	 \end{split}
\end{equation}

The $\Delta_1$ term is the contribution of the siblings of $x$ to the
logarithmic potential at $x$, and as $a_{n-1}(p)$ varies, it acts as a corrective term to the error $e_{n-1}(p)$ from the previous generation. Indeed, suppose that $K_{n-1}$ has been constructed so that $\abs{e_{n-1}(p)} \le \eta$. Suppose that $y \ne x$ is a sibling of $x$, and $x$ has index $k$ and $y$ has index $\ell$. It follows that 
$|x-y| = a_{n-1}(p)\,r^{n-1}\,|q_k-q_\ell|$. Recalling the definition $b_k=\sum_{\ell\ne k}\log|q_k-q_\ell|$, one computes
\begin{equation}\label{eqDelta1}
\begin{split}
 \Delta_1 &=  \sum_{\ell\ne k} \log\left( a_{n-1}(p)r^{n-1}|q_k-q_\ell| \right)
= \sum_{\ell\ne k} \left( \log a_{n-1}(p) +(n-1)\log r +\log|q_k-q_\ell| \right) \\& = (N-1)\left(\log a_{n-1}(p) +(n-1)\log r\right) +b_k.
\end{split}
\end{equation}

Suppose $p$ has index $j$. Then we define $a_{n-1}(p)$ by the equality
\[
 \log a_{n-1}(p)=\frac{N}{N-1}
 \bigl(\eta-\beta_j-e_{n-1}(p)\bigr),
\]
where $\beta_j = b_j - \max_{i} b_i$ is as in the equations preceding Lemma \ref{lem:cheb0}.
Then, since $\beta_j\leq0$ and $|e_{n-1}(p)|\leq\eta$,  it holds that $a_{n-1}(p)\ge1$. Also, using that $\beta_j=\log\frac{\sin\theta_1}{\sin\theta_j}$,
\[
 a_{n-1}(p)\lesssim_\eta
 \left(\frac{\sin\theta_j}{\sin\theta_1}\right)^{N/(N-1)}
 \lesssim_\eta N\sin\theta_j,
\]
taking into account that $\theta_1\simeq 1/N$ and $N^{N/(N-1)}\lesssim N$.
Moreover,
\[
 N(\beta_j+e_{n-1}(p))+(N-1)\log a_{n-1}(p)=N\eta.
\]
Hence, by induction, \eqref{eqDelta1},  the choice of $a_{n-1}(p)$, and the fact that $\beta_k=b_k-b_*$,
\begin{align*}
g_n(x) &= g_{n-1}(p) +N^{-n}(\Delta_1+\Delta_2+\Delta_3)\\
 & = 
c_{n-1}+N^{-n+1}\bigl(\beta_{j}+e_{n-1}(p)\bigr) \\
&\quad+N^{-n}((N-1)\left(\log a_{n-1}(p) +(n-1)\log r\right) +b_k+\Delta_2+\Delta_3)\\
&
=c_n+N^{-n}(\beta_k+\Delta_2+\Delta_3),
\end{align*}
where
\[
 c_n=c_{n-1}+N^{-n}\bigl[(N-1)(n-1)\log r+b_*+N\eta\bigr].
\]
Notice that $\abs{c_{n} - c_{n-1}} \le C_{r, \eta, b^*} nN^{-n+1}$, so $c_n$ converges as $n \rightarrow \infty$. Moreover, the new error is $e_n(x)=\Delta_2+\Delta_3$. Hence to prove Lemma \ref{lem:basic-ind}, we need only show the following estimates:
\vv

\begin{lemma}[Basic error estimate]\label{lem:basic-error}
Suppose that $a_{k}(p)$ are chosen so that $1 \le a_k(p) \lesssim_\eta N \sin(\theta_{j(p)})$ for $1 \le k \le n-1$. Then for all $N$ sufficiently large, we have 
\[
 |\Delta_2|\lesssim_\eta N^3r+N^6r^2,
 \qquad
 |\Delta_3|\lesssim_\eta N^5r^2+N^9r^4.
\]
\end{lemma}
\begin{proof}
For $\ell \ge 1$, recall the notation $\mathcal Y_\ell(p)$, which denotes the set of points
$y\in K_{n-1}$ whose nearest common ancestor with $p$ is $\ell$ generations above $p$. We split 
\begin{align*}
	\Delta_2 & = N \sum_{y \in K_{n-1} \setminus\{p\}} \left( \log\abs{x-y} - \log\abs{p-y}\right) =  \sum_{ \ell = 1  }^{n-1} \Delta_2^{(\ell)}, \\
	\Delta_2^{(\ell)} & \coloneqq N \sum_{  y \in \mathcal{Y}_\ell(p) } \left( \log\abs{x-y} - \log\abs{p-y}\right).
\end{align*}
Recall from Lemma \ref{lem:ahlf_reg} that
\begin{equation}\label{eq:corona-cardinality-delta2}
	\#\mathcal Y_\ell(p)
	=
	(N-1)N^{\ell-1}.
\end{equation}
Let us first estimate $\Delta_2^{(1)}$, which corresponds to the siblings of $p$ in $K_{n-1}$. 

Suppose that \(p\) has index $j$ and parent \(q\in K_{n-2}\), and recall that \(x\in K_n\) has index $k$ and parent \(p\), so that we may write
\begin{align*}
p =q+a_{n-2}(q)r^{n-2}q_j,  \qquad  x=p+a_{n-1}(p)r^{n-1}q_k.
\end{align*}
The elements of $\mathcal{Y}_1(p)$ are given by $y_\ell = q + a_{n-2}(q)r^{n-2} q_\ell$ for $\ell \ne j$, and so direct computation gives us 
\begin{align*}
	\frac{ \abs{x-y_\ell} }{\abs{p - y_\ell}} = \abs{ 1 + r \frac{a_{n-1}(p)}{a_{n-2}(q)}  \frac{ q_k  }{q_j - q_\ell}  } \eqqcolon \abs{ 1 +  \frac{ t_x  }{q_j - q_\ell}  }.
\end{align*}
Since $a_{n-2}(q)\ge1$, $|q_k|\le1$, and $a_{n-1}(p)\lesssim_\eta N\sin\theta_j$, we see
\begin{equation}\label{eq:t-estimate} |t_x| \lesssim_\eta rN\sin\theta_j,
\end{equation} 
and so with the estimate \eqref{eqTX23} we further obtain
\begin{align*}
	\abs{\frac{t_x}{q_j - q_\ell}  } \lesssim_\eta r N^3 \sin\theta_j \ll 1,
\end{align*}
provided that $N$ is sufficiently large. We may then use $\log\abs{1 + u} = u + O(u^2)$ for $\abs{u} \le 1/2$ to obtain 
\begin{align*}\label{eq:Delta2-first-exact} \Delta_2^{(1)} =  N\sum_{y \in \mathcal{Y}_1(p)}
	\log\frac{|x-y|}{|p-y|} & = N\sum_{\ell\ne j} \log\left| 1+\frac{t_x}{q_j-q_\ell} \right|  \\
	& = Nt_x\sum_{\ell\ne j}\frac1{q_j-q_\ell} 
	+ O\left( Nt_x^2 \sum_{\ell\ne j} \frac1{|q_j-q_\ell|^2} \right).
\end{align*} 
By applying Lemma \ref{lem:cheb}, \eqref{eq:t-estimate}, and recalling that $(1 - q_j^2) = \sin^2 \theta_j$,  we deduce that 
\begin{equation}\label{eq:Delta2-first-final}
	\begin{split}
	\abs{  \Delta_2^{(1)}  } & \le N \abs{t_x} \abs{\frac{q_j}{2(1- q_j^2)}} + C N^3 \abs{t_x}^2/(\sin \theta_j)^2 \\
	& \lesssim_\eta r N^2/(\sin \theta_j) + r^2 N^5 \lesssim N^3r+N^5r^2,
	\end{split}
\end{equation}
since $\sin\theta_j \ge \sin \theta_1 \simeq N^{-1}$.

We now estimate the contribution of the $\mathcal{Y}_\ell(p)$ for $\ell \ge 2$, which is much simpler. Fix $y \in \mathcal{Y}_\ell(p)$, and recall from Lemma \ref{lem:ahlf_reg} that for $N$ large enough, we have 
\begin{align*}
	\abs{p-y} \gtrsim N^{-2} r^{n-\ell-1},
\end{align*}
while we have for $h = x-p$ that $\abs{h} \lesssim_\eta N r^{n-1}$. Hence we may again expand $\log\abs{1+ u} \simeq u$ for $u$ small to see 
\begin{align*}
	\abs {\log \frac{\abs{x-y}}{\abs{p-y}}} = \abs{\log \abs{  1 + \frac{h}{p-y}  }} \lesssim_\eta   N^3r^\ell, 
\end{align*}
where the expansion is justified since $\ell \ge 2$ and $N^3 r^2 \ll 1$ for $N$ large. Altogether, we estimate using the above inequality along with \eqref{eq:corona-cardinality-delta2},
\begin{align*}
	\abs{ \sum_{\ell = 2}^{n-1} \Delta_2^{\ell} } \le N \sum_{\ell=2}^{n-1} \sum_{y \in \mathcal{Y}_\ell(p)} \abs{ \log \frac{\abs{x-y}}{\abs{p-y}}  } \lesssim_\eta  N \sum_{\ell=2}^{n-1} N^3 r^\ell N^\ell \lesssim N^6 r^2 \ll 1,
\end{align*}
since $Nr \ll 1$ for $N$ large. Along with \eqref{eq:Delta2-first-final} this completes the proof of the estimate on $\Delta_2$.

\vspace{1em}

Now, we move to the estimates for $\Delta_3$. Recall that
\[
\Delta_3
=
\sum_{\substack{y\in K_{n-1}\\ y\ne p}}
\sum_{\widehat z=y}
\bigl(
\log|x-z|-\log|x-y|
\bigr).
\]
For $y\in K_{n-1}\setminus\{p\}$, the $N$ children of $y$ are
\[
z_m
=
y+a_{n-1}(y)r^{n-1}q_m,
\qquad 1\le m\le N,
\]
and so the contribution of the children of $y$ to $\Delta_3$ can be
written as
\begin{align}
\sum_{\widehat z=y}
\bigl(
\log|x-z|-\log|x-y|
\bigr)
&=
\sum_{m=1}^N
\log\left|
1-\frac{a_{n-1}(y)r^{n-1}q_m}{x-y}
\right|.
\label{eq:Delta3-one-parent}
\end{align}
We claim that $a_{n-1}(y) r^{n-1}q_m / (x-y)$ is sufficiently small, so that we can use the expansion $\log\abs{1 + u} = u + O(u^2)$ in each logarithm above.

Indeed, if $y \in \mathcal{Y}_\ell(p)$ for some $1 \le \ell \le n-1$, then we may apply Lemma \ref{lem:ahlf_reg} to see that 
\begin{align}\label{eqn:x-y_low}
	\abs{x-y} \ge \abs{y-p} - \abs{x-p} \gtrsim_\eta N^{-2} r^{n-1-\ell} - a_{n-1}(p) r^{n-1} \ge r^{n-1-\ell} \left(  N^{-2} - C_\eta N r^\ell \right).
\end{align}
Since $s < 1/3$ and $r = N^{-1/s}$, then as long as $N$ is chosen large enough, the difference in the last line can be made positive, and we have that $\abs{x-y} \gtrsim_\eta r^{n-1-\ell} N^{-2}$. On the other hand, we have that $\abs{a_{n-1}(y) r^{n-1} q_m} \lesssim_\eta N r^{n-1}$, and so 
\begin{align*}
	\abs{ \frac{a_{n-1}(y) r^{n-1} q_m}{x-y}  } \lesssim_\eta N^3 r^\ell \le N^3 r \ll 1, 
\end{align*}
again since $0 < s < 1/3$. Expanding the logarithm in \eqref{eq:Delta3-one-parent} and using the identity $\sum_{m=1}^Nq_m=0$ to cancel the linear term in $q_m$, we see that 
\begin{align}
	\left|
	\sum_{\widehat z=y}
	\bigl(
	\log|x-z|-\log|x-y|
	\bigr)
	\right|
	&\lesssim
	\sum_{m=1}^N
	\frac{a_{n-1}(y)^2r^{2n-2}q_m^2}{|x-y|^2}
	\lesssim
	N\frac{a_{n-1}(y)^2r^{2n-2}}{|x-y|^2}.
	\label{eq:Delta3-one-parent-bound}
\end{align}
As in the estimates for $\Delta_2$, we split our estimates by generations $\mathcal{Y}_\ell(p)$ for $\ell =1$ and $2 \le \ell \le n-1$; that is, we write 
\begin{align*}
	\Delta_3 = \sum_{\ell=1}^{n-1} \Delta_3^{(\ell)}, \qquad \Delta_3^{(\ell)} = \sum_{y \in \mathcal{Y}_\ell(p)} \sum_{\hat{z} = y} \left(  \log \abs{x-z} -\log \abs{x-y}  \right).
\end{align*}

For $\ell =1$, suppose that $p$ has index $j$ and parent $q \in K_{n-2}$, so that $p = q + a_{n-2}(q) r^{n-2} q_j$. The other
children of $q$ (i.e., the elements of $\mathcal{Y}_1(p)$) are $y_i
=
q+a_{n-2}(q)r^{n-2}q_i$ for $i\ne j$. 
Since
\[
p-y_i
=
a_{n-2}(q)r^{n-2}(q_j-q_i), \text{  and  } 
x-p=O_\eta(Nr^{n-1}),
\]
the smallness assumption $N^3r\ll_\eta1$ implies
\[
|x-y_i|\approx |p-y_i|
\approx
a_{n-2}(q)r^{n-2}|q_j-q_i|,
\]
with constants independent of $N$. 
Therefore, using
$a_{n-2}(q)\ge1$, \eqref{eq:Delta3-one-parent-bound} yields
\begin{align}
	|\Delta_3^{(1)}|
	&\lesssim
	N
	\sum_{i \ne j}
	\frac{a_{n-1}(y_i)^2r^{2n-2}}
	{a_{n-2}(q)^2r^{2n-4}|q_j-q_i|^2}
	\lesssim
	Nr^2
	\sum_{i \ne j}
	\frac{a_{n-1}(y_i)^2}{|q_j-q_i|^2}.
	\label{eq:Delta3-first-pre}
\end{align}
Since $y_i$ has index $i$, we have the bound $a_{n-1}(y_i)
\lesssim_\eta
N\sin\theta_i$ and thus 
\begin{align}
	\sum_{i\ne j}
	\frac{a_{n-1}(y_i)^2}{|q_j-q_i|^2}
	&\lesssim_\eta
	N^2
	\sum_{i\ne j}
	\frac{\sin^2\theta_i}{|q_j-q_i|^2} \lesssim_\eta N^4,
\end{align}
where the last inequality follows from Lemma~\ref{lem:weighted}, with $m=a=2$. Substituting into \eqref{eq:Delta3-first-pre}, we obtain
\begin{equation}\label{eq:Delta3-first-final}
	|\Delta_3^{(1)}|
	\lesssim_\eta
	N^5r^2.
\end{equation}

The estimates for $\Delta_3^{(\ell)}$ with $\ell \ge 2$, as in the case of $\Delta_2$, are much simpler. We use the previous bound $\abs{x-y} \gtrsim_\eta N^{-2} r^{n-1-\ell}$ from \eqref{eqn:x-y_low}, the fact that $\# \mathcal{Y}_\ell(p) = (N-1)N^{\ell-1} \le N^\ell$, and the crude upper bound $\abs{a_{n-1}(y)} \lesssim_\eta N$ to pass from \eqref{eq:Delta3-one-parent-bound} to
\begin{align*}
	\sum_{\ell = 2}^{n-1} |\Delta_3^{(\ell)}| \lesssim \sum_{\ell=2}^{n-1} \sum_{y \in \mathcal{Y}_\ell (p)} N \frac{  a_{n-1}(y)^2 r^{2n-2}  }{\abs{x-y}^2} \lesssim N^7 \sum_{\ell=2}^{n-1} (N r^2)^{\ell } \lesssim N^9 r^4,
\end{align*}
where again the summability of the geometric series comes from the fact that $N r^2 \ll 1$. Combining the above with \eqref{eq:Delta3-first-final}, we conclude 
\begin{align*}
	\abs{\Delta_3} \lesssim N^5 r^2 + N^9 r^4
\end{align*}
as desired.
\end{proof}
\vv

We can now prove the main lemma of this section.
\begin{proof}[Proof of Lemma~\ref{lem:basic-ind}]
	
{Fix $\eta >0$ small. 
	At generation one, choosing $a_0(0) = 1$, we may write for $x = q_k \in K_1$,
	\begin{align*}
		g_1(x) = g_1(q_k) = N^{-1} \sum_{j \ne k} \log \abs{q_k -q_j} = \frac{b_k}N = \frac{b^*}N + \frac{\beta_k}N.
	\end{align*}
	and so $e_1(x) = e_1(q_k)=0$.
Of course the first choice, $a_0(0) =1$ satisfies $1 \le a_0(0) \lesssim_\eta N$. 
	
	By Lemma \ref{lem:basic-error} and the argument that precedes it, given that the error $e_{n-1}(x)$ satisfies $\abs{e_{n-1}(x)} < \eta$ for all $x \in K_{n-1}$, inductively we may choose $a_{n-1}(p)$ satisfying $1 \le a_{n-1}(p) \lesssim_\eta N \sin \theta_{j(p)}$ so that 
	\begin{align*}
		\abs{e_{n}(x)} \le \abs{\Delta_2} + \abs{\Delta_3} \lesssim_\eta N^3 r < \eta, \text{ for all } x \in K_n,
	\end{align*}
	provided that $N$ is large enough, since $s < 1/3$.

}
\end{proof}
\vv

The next theorem now follows from Lemmas \ref{lem:ahlf_reg} and \ref{lem:basic-ind} in a standard way.

\begin{theorem}\label{thm:one-third}
	For every $0 < s< 1/3$, there exists an $s$-Ahlfors regular compact set $K \subset \R \subset \C$ so that the harmonic measure in $\C \setminus K$ with pole at infinity $\omega^\infty$ is equivalent to $\mathcal{H}^s|_K$, i.e., for some constant $C \ge 1$, 
	\begin{align*}
		C^{-1} \, \mathcal{H}^s|_K(F)\le \omega^\infty(F) \le C \mathcal{H}^s|_K(F) \text{ for all } F \subset K \text{ Borel.}
	\end{align*}
	In particular, $\dim \omega^\infty = s$. 
\end{theorem}
\begin{proof}
	Since the proof of this Theorem follows from the arguments in \cite{DJJ23}, we outline just the steps. As long as $N$ is chosen large enough, Lemma \ref{lem:ahlf_reg} shows that $K_n$ converge in the Hausdorff metric to a compact, $s$-Ahlfors regular subset $K \subset \R\subset \C$. Moreover, if $\mu$ is the weak limit of the Radon measures
	\begin{align*}
		\mu_n = N^{-n} \sum_{x \in K_n} \delta_x, 
	\end{align*}
	then we know that $\mu$ is an $s$-Ahlfors regular probability measure supported on $K$, and thus is equivalent to $\mathcal{H}^s|_K$. Arguing just as in \cite[Lemma 3.3]{DJJ23}, but now using Lemma \ref{lem:basic-ind}, one can show that the potential
	\begin{align*}
		g_\mu(x) = \int \log \abs{ x-y}  d\mu(y) 
	\end{align*}
	is constant on $K$, and so by uniqueness, $\mu$ is the equilibrium measure for $K$. Since $K$ is $s$-Ahlfors regular with $s > 0$, we know that $\C \setminus K$ is Wiener regular, and so the equilibrium measure $\mu = \omega_{\C\setminus K}^\infty$ is exactly the harmonic measure with pole at infinity. Since $\mu$ is equivalent to $\mathcal{H}^s|_K$, this completes the proof. 
\end{proof}

\section{First higher-order corrections: an overview}
\label{sec:first-corrections}

The construction of Section \ref{sec:basic-construction} uses, at each point $p\in K_{n-1}$, a
single parameter $a_{n-1}(p)$: all $N$ children of $p$ are placed at
the rescaled Chebyshev positions
\[
p+a_{n-1}(p)r^{n-1}q_k,
\qquad 1\le k\le N.
\]
The parameter $a_{n-1}(p)$ is chosen so as to correct the oscillation error
in the logarithmic potential inherited from the previous generation, which is enough to prove our main result when $s < 1/3$.

The reason it fails beyond $1/3$ is that,
even after this error has been corrected, the potential
created by the rest of the set is not exactly constant across the $N$
children of $p$. Its first nonconstant term is approximately linear in
the variable $q_k$. The purpose of the present section is to
explain how to modify the relative positions of the children in order
to cancel this linear variation. Once this is done, the remaining
nonconstant error is quadratic, and a second correction cancels that
quadratic profile.

The calculations in this section are meant mainly to motivate the
general polynomial renormalization developed later in Section \ref{sec:poly_corrections}. We therefore
emphasize the geometry of the construction and only sketch the error
estimates.

\subsection{The first correction: a quadratic deformation}

Fix $p\in K_{n-1}$. In the construction of Section~\ref{sec:basic-construction}, the children of
$p$ are located at
\[
    p+a_{n-1}(p)r^{n-1}q_k,
    \qquad 1\leq k\leq N.
\]
To go beyond the range $s<1/3$, we deform the relative positions of
the children. We retain the correction mechanism for the coefficients $a_{n-1}(p)$ of the basic construction, but introduce an additional deformation of the relative positions of the children.
The only new modification is in the relative positions
of the children: the Chebyshev nodes $q_k$ are replaced by a slightly
deformed configuration. Set
\[
    \phi(t)=t^2-\frac12,
    \qquad
    Q_p^{(1)}(t)=t+\delta_p\phi(t),
\]
and define
\begin{equation}\label{eq:first-corrected-children}
    x_k
    =
    p+a_{n-1}(p)r^{n-1}Q_p^{(1)}(q_k),
    \qquad 1\leq k\leq N.
\end{equation}
Since
\[
    \frac1N\sum_{k=1}^N q_k^2=\frac12,
\]
we have
\[
    \sum_{k=1}^N \phi(q_k)=0
    \qquad\text{and hence}\qquad
    \sum_{k=1}^N Q_p^{(1)}(q_k)=0.
\]
Thus the deformation preserves the barycenter of each family of
children, which is the cancellation needed in the estimate of
$\Delta_3$.

Moreover, provided $|\delta_p|$ is small,
\[
    (Q_p^{(1)})'(t)=1+2\delta_pt
\]
is bounded above and below by positive constants, so
\[
    |Q_p^{(1)}(q_k)-Q_p^{(1)}(q_\ell)|
    \simeq |q_k-q_\ell|.
\]
Hence the geometric separation of the Chebyshev configuration is
preserved.

For a polynomial $P$, define
\begin{equation}\label{defLP}
    (LP)(q_k)
    =
    \sum_{\ell\neq k}
    \frac{P(q_k)-P(q_\ell)}{q_k-q_\ell}.
\end{equation}
The operator $L$ describes the first-order variation of the logarithmic
interaction among siblings.  Indeed, under the perturbation
\[
    q_k\mapsto q_k+\varepsilon P(q_k),
\]
we have
\[
    \left.
    \frac{d}{d\varepsilon}
    \right|_{\varepsilon=0}
    \sum_{\ell\neq k}
    \log
    \left|
        q_k+\varepsilon P(q_k)
        -q_\ell-\varepsilon P(q_\ell)
    \right|
    =
    (LP)(q_k).
\]
For $\phi(t)=t^2-\frac12$, one computes
\[
    L\phi(q_k)=(N-2)q_k.
\]
Thus a quadratic deformation produces, to first order, a linear
profile in the Chebyshev node $q_k$.

We now compare this with the other terms. Set $\rho_p=a_{n-1}(p)r^{n-1}$, and define
\[
    H_p(h)
    =
    N\sum_{\substack{y\in K_{n-1}\\ y\neq p}}
    \bigl(\log|p+h-y|-\log|p-y|\bigr),
\]
so that a simple computation yields $\Delta_2(k)
=
H_p\bigl(\rho_p Q_p^{(1)}(q_k)\bigr)$. Writing
$
    \Lambda_p=\rho_pH_p'(0),
$
the first-order Taylor term is
\[
    \Lambda_p Q_p^{(1)}(q_k)
    =
    \Lambda_p q_k
    +
    \Lambda_p\delta_p\phi(q_k).
\]
Since $|\delta_p|\ll1$, the second term is of the next order.
Thus the leading nonconstant term in $\Delta_2$ is
$
    \Lambda_p q_k.
$

The discrete field estimates proved later in Section~\ref{sec:higher-potential} give
\[
    |\Lambda_p|\lesssim_\eta N^3r.
\]
We therefore choose
\[
    \delta_p=-\frac{\Lambda_p}{N-2},
\]
so that $|\delta_p|\lesssim_\eta N^2r$. In particular, since $r=N^{-1/s}$, we have $|\delta_p|=o(1)$ for every $s<1/2$.

We now estimate the error which remains after the linear cancellation.
Recall that
\[
    \rho_p=a_{n-1}(p)r^{n-1}.
\]
For the sibling term,
\[
    \Delta_1(k)
    =
    (N-1)\log\rho_p
    +
    \sum_{\ell\neq k}
    \log
    \left|
        Q_p^{(1)}(q_k)-Q_p^{(1)}(q_\ell)
    \right|.
\]
Since
\[
    Q_p^{(1)}(q_k)-Q_p^{(1)}(q_\ell)
    =
    (q_k-q_\ell)
    \bigl(1+\delta_p(q_k+q_\ell)\bigr),
\]
we obtain
\[
    \Delta_1(k)
    =
    (N-1)\log\rho_p+b_k
    +
    \sum_{\ell\neq k}
    \log\bigl|1+\delta_p(q_k+q_\ell)\bigr|,
\]
where
\[
    b_k=\sum_{\ell\neq k}\log|q_k-q_\ell|
\]
is the Chebyshev contribution already present in the basic construction.

Since $|\delta_p|\lesssim_\eta N^2r=o(1)$, we have
\[
    |\delta_p(q_k+q_\ell)|\leq 2|\delta_p|\ll1,
\]
and hence we may Taylor expand. Therefore
\[
    \Delta_1(k)
    =
    (N-1)\log\rho_p+b_k
    +(N-2)\delta_p q_k
    +O_\eta(N^5r^2).
\]
Indeed, the quadratic and higher-order terms are bounded by
\[
    O\bigl(N|\delta_p|^2\bigr)
    =
    O_\eta(N^5r^2).
\]

For the other terms, write
\[
    \Delta_2(k)
    =
    N\sum_{\substack{y\in K_{n-1}\\ y\neq p}}
    \log\left|
        1+\frac{x_k-p}{p-y}
    \right|.
\]
The Taylor expansion is justified by the geometric estimates of
Section~\ref{sec:corrected-geometry}. Indeed,
\eqref{eq:relative-scale} and the uniform boundedness of
$Q_p^{(1)}$ give
\[
    \left|\frac{x_k-p}{p-y}\right|
    \lesssim_\eta \tau=o(1),
    \qquad y\in K_{n-1}\setminus\{p\}.
\]
Hence the logarithm may be expanded uniformly.

With
\[
    \Lambda_p
    =
    N\rho_p
    \sum_{\substack{y\in K_{n-1}\\ y\neq p}}
    \frac{1}{p-y},
\]
the first-order Taylor contribution is
\[
    \Lambda_p Q_p^{(1)}(q_k)
    =
    \Lambda_p q_k
    +
    \Lambda_p\delta_p\phi(q_k).
\]
Since
\[
    |\Lambda_p|\lesssim_\eta N^3r,
    \qquad
    |\delta_p|\lesssim_\eta N^2r,
\]
the second term is bounded by
\[
    O_\eta(N^5r^2).
\]
The second-order Taylor contribution is also of size
$O_\eta(N^5r^2)$. Hence
\[
    \Delta_2(k)
    =
    \Lambda_p q_k
    +
    O_\eta(N^5r^2).
\]

We choose
\[
    \delta_p=-\frac{\Lambda_p}{N-2},
\]
so that
\[
    (N-2)\delta_p+\Lambda_p=0.
\]
Thus the leading linear terms in $\Delta_1+\Delta_2$ cancel,
and the remaining nonconstant part of $\Delta_1+\Delta_2$
is of size
$
    O_\eta(N^5r^2).
$

Finally, the quadratic deformation preserves the barycenter of every
family of children and satisfies the geometric hypotheses of
Section~\ref{sec:corrected-geometry}.  The cluster estimate proved
later in Lemma~\ref{lem:cluster-profiles-new}, specialized to this
quadratic deformation, therefore gives
\[
    |\Delta_3|
    \lesssim_\eta
    N^5r^2+N^9r^4.
\]
Consequently, after the quadratic correction the total new error is
of size
\[
    O_\eta(N^5r^2).
\]
Since $N^5r^2=N^{5-2/s}\rightarrow0$ as $N \rightarrow \infty$ whenever $s<2/5$, this first correction extends the construction to
the range $s<2/5$.

\subsection{The second correction: a cubic deformation}

The next error is no longer linear in the child variable. After the
quadratic deformation has been introduced and the linear profile has
been cancelled, the leading nonconstant term is quadratic in $q_k$.
Up to an additive constant, it is proportional to
\[
q_k^2-\frac12.
\]

To cancel this profile, we add a second deformation, now of cubic
degree.
Define
\[
\psi(t)
=
t^3-\frac{2N-3}{2(N-1)}t.
\]
This polynomial has zero average on the Chebyshev nodes,
\[
\sum_{k=1}^N\psi(q_k)=0,
\]
and therefore adding it to the child positions still preserves the
barycenter.

We now set
\[
Q_p^{(2)}(t)
=
t+\delta_p\phi(t)+\varepsilon_p\psi(t),
\]
and define the children by
\begin{equation}\label{eq:second-corrected-children}
x_k
=
p+a_{n-1}(p)r^{n-1}Q_p^{(2)}(q_k).
\end{equation}
The operator $L$ introduced in \eqref{defLP} satisfies
\[
L\psi
=
(N-3)\left(t^2-\frac12\right)
\]
on the Chebyshev nodes.
Thus the cubic deformation produces, to first order, precisely the centered quadratic mode \(t^2-\frac12\). As will be verified in the higher-order estimates below, after the linear profile has been cancelled, the leading nonconstant term of \(\Delta_1+\Delta_2\) is a multiple of this mode.

Denoting the coefficient of this quadratic profile by $\Gamma_p$,
we choose
\[
\varepsilon_p
=
-\frac{\Gamma_p}{N-3}.
\]

The estimates proved later in Section~\ref{sec:higher-potential} give
\[
|\Gamma_p|\lesssim N^5r^2,
\]
and hence
\[
|\varepsilon_p|\lesssim N^4r^2.
\]

After this second correction, the remaining nonconstant error is of size
$O(N^7r^3)$.  Thus the nonconstant profiles are controlled whenever
\[
N^7r^3\to0,
\]
which corresponds to
\[
s<\frac37.
\]
At this stage there may still be parent-dependent constant profiles of
lower order.  These constants are retained, rather than discarded, and
are cancelled in the full renormalization argument of
Section~\ref{sec:finite-renormalization}.

\subsection{What is modified at each generation}

It is useful to summarize the recursive construction. Suppose that $K_{n-1}$ has already been constructed. For each
$p\in K_{n-1}$, the geometry of the previous generations is already
fixed. From this geometry one computes the derivatives of the discrete
external potential at $p$ up to the order needed in the construction;
see Section~\ref{sec:higher-potential}.

One then chooses:

\begin{itemize}
\item a scale factor $a_{n-1}(p)$, which corrects the scalar part of
the potential;
\item a quadratic coefficient $\delta_p$, which corrects the linear
variation across the children;
\item if desired, a cubic coefficient $\varepsilon_p$, which corrects
the quadratic variation;
\item and, in the general construction, further polynomial correction
coefficients which cancel the higher-order profiles.
\end{itemize}
Once these parameters have been chosen, the $N$ children of $p$ are
placed at
\[
p+a_{n-1}(p)r^{n-1}Q_p(q_k),
\]
where $Q_p$ is the corresponding small polynomial deformation of the
identity map, chosen so as to preserve the geometric separation of the
Chebyshev configuration.

The construction involves two nested inductive procedures.
When defining generation $n$, the geometry of $K_{n-1}$ is already
fixed; within the new generation, the polynomial corrections are then
chosen successively according to their order.  At a given order, the
profile to be cancelled depends only on $K_{n-1}$ and on the
lower-order corrections already chosen throughout the same generation.
There is therefore no dependence on corrections of the current or
higher orders.

The examples $s<2/5$ and $s<3/7$ illustrate the general principle:
a polynomial correction of degree $m+1$ is used to cancel the
nonconstant potential profile of degree $m$. The general version of
this mechanism is developed in the next sections.

\section{The polynomial corrections}\label{sec:poly_corrections}

We now formalize the mechanism illustrated in the previous section. In particular, we now fix $s \in (0,1/2)$, define $r$ by $Nr^{s} = 1$ with $N \gg 1$, set $K_0 = \{0\}$, and assume that $K_{m}$ is constructed from $K_{m-1}$ by replacing each $z \in K_{m-1}$ by its $N$ children
\begin{align*}
	z_k = z + a_{m-1}(z) r^{m-1} Q_z(q_k), \qquad 1 \le k \le N,
\end{align*}
where $m$ ranges over $1 \le m \le n$. We shall soon narrow down some restrictions on the size of the $a_{m-1}$ and the deformations $Q_z$ (in fact, each $Q_z$ is a polynomial), but the basic observation is that for $x \in K_n$, the three terms in the difference $g_{n}(x) - g_{n-1}(\hat{x})$
\[
\Delta_1,\qquad \Delta_2,\qquad \Delta_3
\]
as defined in \eqref{eqn:delta} play distinct roles in the oscillation error for $g_n(x) - g_{n-1}(\hat{x})$.

The term $\Delta_2$ measures the change of the potential
when the point of evaluation moves from a parent $p$ to one of its
children $x$. The term $\Delta_3$ measures the error of the potential at $x$ made by replacing a parent $q \in K_{n-1}$, $q \ne \hat{x}$, by its $N$ children. After Taylor expansion, $\Delta_2$ and $\Delta_3$ produce a polynomial
errors in the variable $q_k$, depending on  the polynomial deformations $Q_z$.

The term $\Delta_1$ is the contribution of the potential coming from the interaction between a point $x$ and its siblings. By changing the relative position of the children, it is this term which serves to correct the errors induced by $\Delta_2$ and $\Delta_3$. In particular, the  perturbations of the nodes $q_k$ will be chosen precisely so that the polynomial variation they produce in $\Delta_1$ cancels the lower order polynomial errors of $\Delta_2$ and $\Delta_3$.

In what follows, we first introduce the polynomial correction operator $L$, and show how it appears naturally in the linearization of $\Delta_1$ (Section \ref{sec:polynomial-operator}). Assuming the sets $K_1, \dotsc K_n$ have been constructed with natural size restrictions on the $a_{m-1}(z)$ and the deformations $Q_z(q_k)$, we collect estimates on how the errors $\Delta_1, \Delta_2$, and $\Delta_3$ depend on the polynomial perturbations $Q_p$. The main result here is Lemma \ref{lem:higher-basic-error}.

\subsection{The polynomial correction operator, $L$}
\label{sec:polynomial-operator}

Let
\[
q_1,\ldots,q_N
\]
be the Chebyshev nodes. If $P$ is a polynomial, we initially define
$LP$ as a function on these nodes by
\begin{equation}\label{eq:def-L}
	(LP)(q_k)
	:=
	\sum_{\substack{\ell=1\\ \ell\ne k}}^N
	\frac{P(q_k)-P(q_\ell)}{q_k-q_\ell},
	\qquad 1\le k\le N.
\end{equation}
The reason for introducing this operator is that it is the
linearization of the ``sibling'' logarithmic potential, i.e., the term  $\Delta_1$. Indeed, suppose we perturb the nodes $q_k$ by 
\[
q_k\longmapsto q_k+\varepsilon P(q_k).
\]
The logarithmic potential at the $k$-th node, $q_k$, induced by its siblings
is
\[
\Phi_k(\varepsilon)
=
\sum_{\ell\ne k}
\log\left|
q_k+\varepsilon P(q_k)
-q_\ell-\varepsilon P(q_\ell)
\right|.
\]
Differentiating at $\varepsilon=0$, we obtain
\[
\Phi_k'(0)
=
\sum_{\ell\ne k}
\frac{P(q_k)-P(q_\ell)}{q_k-q_\ell}
=
(LP)(q_k).
\]
Thus, to first order, we have
\begin{align*}
	\Phi_k(\epsilon) \approx \sum_{\ell \ne k} \log \abs{q_k - q_\ell} + \epsilon (LP)(q_k).
\end{align*}
We shall soon see that we can expand oscillations coming from $\Delta_2$ and $\Delta_3$ as a polynomial $R(q_k)$ in the Chebyshev nodes $q_k$. Our main goal is then to carefully choose $P$ and $\epsilon$ so as to cancel these $R(q_k)$, while still maintaining the zeroth term in the expansion of $\Phi_k(\epsilon)$ as above. This will require us to have very precise estimates on how small polynomial perturbations of the nodes $q_k$ affect the geometry of the set, and the error terms $\Delta_i$, which are forthcoming. First, though, we need the following facts about the operator $L$. 

The crucial \textit{algebraic} property of $L$ is that it lowers polynomial
degree by one. For example,
\[
L(t^2-\tfrac12)=(N-2)t,
\]
which is the identity used in the quadratic correction, and the cubic
polynomial introduced in Section~\ref{sec:first-corrections} satisfies
\[
L\psi=(N-3)(t^2-\tfrac12).
\]
More generally, if $d<N$, then
\begin{equation}\label{eq:L-triangular}
	L(t^d)
	=
	(N-d)t^{d-1}
	+\sum_{\nu=1}^{d-1} S_\nu t^{d-1-\nu},\qquad S_\nu = \sum_{\substack{\ell=1}}^Nq_\ell^\nu.
\end{equation}
Indeed,
\begin{equation}\label{eq:L-triangular22}
	\frac{x^d-y^d}{x-y}
	=
	\sum_{\nu=0}^{d-1}x^{d-1-\nu}y^\nu,
\end{equation}
and thus
$$L(t^d)(q_k) =
\sum_{\substack{\ell=1\\ \ell\ne k}}^N \sum_{\nu=0}^{d-1}q_k^{d-1-\nu}q_\ell^\nu = 
\sum_{\nu=0}^{d-1}q_k^{d-1-\nu} (S_\nu - q_k^\nu) = (N-d) q_k^{d-1} + \sum_{\nu=1}^{d-1} S_\nu q_k^{d-1-\nu},
$$
which proves 
\eqref{eq:L-triangular}. Notice also that $|S_\nu|\leq N$.

Consequently, if $N>M+1$, then for every $m\le M$ the operator
$$
L:\mathcal P_{m+1}/\mathbb R\longrightarrow\mathcal P_m
$$
is an isomorphism. Here $\mathcal P_m$ stands for the space of polynomials of degree at most $m$. Moreover, since $M$ is fixed, its inverse satisfies
a quantitative estimate. Indeed, if
$$
R(t)=\sum_{j=0}^m a_jt^j,\qquad
P(t)=\sum_{j=1}^{m+1}b_jt^j,
\qquad LP=R,
$$
then comparison of the coefficients from the highest degree downwards,
using \eqref{eq:L-triangular} and the fact that all the lower-order
coefficients of $L(t^d)$ are $O_M(N)$, gives, for $N\ge2(M+1)$,
$$
\|P\|_{\mathrm{coeff}}
\le
\frac{C_M}{N}\|R\|_{\mathrm{coeff}},
$$
where
$
\|\cdot\|_{\mathrm{coeff}}$ stands for the $\ell^\infty$ norm of the coefficients.
Since $\deg P\le M+1$,
$$
\|P\|_{C^1([-1,1])}
\lesssim_M\|P\|_{\mathrm{coeff}},
$$
and therefore
\begin{equation}\label{eq:L-inverse}
	\|P\|_{C^1([-1,1])}
	\le
	\frac{C_M}{N}\|R\|_{\mathrm{coeff}}.
\end{equation}

Since $L$ annihilates constants, we may finally replace $P$ by $
P-\frac1N\sum_{k=1}^NP(q_k)$
and thereby impose that the inverse $LP = R$ satisfies the cancelation condition
$$
\sum_{k=1}^NP(q_k)=0.
$$
The constant that is subtracted has absolute value at most
$\|P\|_{L^\infty([-1,1])}$, so \eqref{eq:L-inverse} remains valid,
up to changing $C_M$. Since also
$$
\sum_{k=1}^Nq_k=0,
$$
this normalization ensures that the deformation
$q_k\mapsto q_k+P(q_k)$ preserves the barycenter of the children.
This property will be important in the estimates for $\Delta_3$. We summarize the aforementioned properties in the following lemma:

\begin{lemma}\label{lem:L_prop}
	Let $ M \in \N$ and let $N \ge 2M+1$. For each $0 \le m \le M$, the operator 
	\begin{align*}
		L: \mathcal{P}_{m+1}/\R \rightarrow \mathcal{P}_m 
	\end{align*}
	is an isomorphism for which $\deg( LP) \le \deg(P) - 1$, and for which the inverse $L(P) = L(P + \mathrm{const}) = R$ for $R \in \mathcal{P}_m$ satisfies
	\begin{align*}
		\| P - N^{-1}\sum_{k=1}^N P(q_k)  \|_{C^1([-1,1])} \le \frac{C_M}{N} \| R \|_{\mathrm{coeff}}.
	\end{align*}
\end{lemma}

\subsection{Estimates on the previous generations, $K_1$, $\dotsc$, $K_{n-1}$}
\label{sec:corrected-geometry}

Fix an integer $M\ge2$, and a small parameter $\eta > 0$. Throughout the rest of the paper we write
\[
    \tau=N^2r,
\]
which is the natural scaling parameter for many of the estimates that follow; for $N \gg 1$, $\tau$ is very small since $Nr^s = 1$ and $ s< 1/2$. All constants implicit in an inequality like $\lesssim_M$ depend only on $M$, while those
implicit in $\lesssim_{M,\eta}$ depend only on $M$ and the fixed
parameter $\eta$.  In particular, these constants are independent of
$N$ and of the generation $n$. There are two iterative procedures in our construction: the induction on
the generation $n$, and the induction on the finite-order polynomial correction scheme.  This subsection
concerns only the first one. 

For every $0\le m\le n-2$ and every $z\in K_m$, assume that the children of $z$
have already been fixed and are of the form
\begin{equation}
\label{eq:old-children}
    z_k=z+\rho_z Q_z(q_k),
    \qquad 1\le k\le N,
\end{equation}
where for convenience, we introduce the notation
\[
    \rho_z=a_m(z)r^m.
\]
We assume that maps $Q_z$ are the deformations that have already been chosen in the previous
generations, and we merely assume the following quantitative information about $Q_z$. First, they are quantitatively close to the identity in the sense that $Q_z \in C^1([-1,1])$ with 
\begin{equation}
\label{eq:Q-close-id}
    \|Q_z'-1\|_{L^\infty([-1,1])}
    \lesssim_M\tau.
\end{equation}
In particular, for $\tau$ sufficiently small they are uniformly
bi-Lipschitz:
\begin{equation}
\label{eq:Q-bilip}
    |u-v|/2
    \le
    |Q_z(u)-Q_z(v)|
    \le
    2|u-v|,
    \qquad
    u,v\in[-1,1].
\end{equation}
Second, they preserve the barycenter of the Chebyshev configuration:
\begin{equation}
\label{eq:Q-barycenter}
    \sum_{k=1}^N Q_z(q_k)=0.
\end{equation}
Notice that \eqref{eq:Q-close-id} and
\eqref{eq:Q-barycenter} for $\tau$ sufficiently small also imply
\begin{equation}
\label{eq:Q-node-bound}
    \| Q_z\|_{L^\infty([-1,1])} \le 2.
\end{equation}
We also assume natural bounds on the scales $a_m(z)$ coming from the previous generations.  If $z\in K_m$, $1\le m\le n-2$, and $z$ has index $j$, then we assume $a_{m}(z)$ has been  chosen so that 
\begin{equation}
\label{eq:old-scale-bound}
    1 \le  a_m(z)
    \lesssim_{M,\eta}N\sin\theta_j.
\end{equation}

The properties \eqref{eq:Q-close-id} and \eqref{eq:Q-barycenter}
will later follow from the size and normalization estimates for the
correction polynomials we choose in an iterative procedure; \eqref{eq:Q-bilip} and \eqref{eq:Q-node-bound} are then immediate consequences of these estimates. With just the estimates above, we can prove the following ``basic'' estimates for the sets $K_m$ with $m \le n-1$. 

\begin{lemma}\label{lem:prev_estimates}
	Suppose $K_m$ has been constructed for $0 \le m \le n-1$ as above, i.e., so that the children of $K_{m-1}$ all satisfy \eqref{eq:old-children}, \eqref{eq:Q-close-id}, and \eqref{eq:Q-barycenter}, and so that the $a_{m}(z)$ satisfy \eqref{eq:old-scale-bound}. 
	
	Then as long as $N$ is sufficiently large, depending just on $s$, $M$, and $\eta$, we have that for all $p \ne y$, $p, y\in K_{n-1}$, 
	\begin{equation}
		\label{eq:relative-scale}
		\frac{\rho_p+\rho_y}{|p-y|}
		\lesssim_{M,\eta}\tau,
	\end{equation}
	and moreover, for every
	integer $d\ge2$ and $0\le a\le d$,
	\begin{equation}
		\label{eq:geometric-sum}
		\sum_{\substack{y\in K_{n-1}\\y\neq p}}
		\frac{\rho_p^{\,d-a}\rho_y^{\,a}}
		{|p-y|^d}
		\lesssim_{d,M,\eta}
		\tau^d.
	\end{equation}
\end{lemma}
\begin{proof}
	We start with the proof of \eqref{eq:relative-scale}. If $p, y \in K_{n-1}$ are siblings with indices $j$ and $\ell$ respectively, we use the estimate on $a_{n-1}(p) + a_{n-1}(y) \lesssim_{M, \eta} N (\sin \theta_j + \sin \theta_\ell)$ and obtain
	\begin{align*}
		\frac{\rho_p + \rho_y}{ \abs{p-y}  } \lesssim_{M, \eta} \frac{ N  (\sin \theta_j + \sin \theta_\ell ) r^{n-1} }{ \abs{Q_q(q_j) - Q_q(q_\ell)}  r^{n-2} } \lesssim_{M, \eta} \frac{N r (\sin \theta_j + \sin \theta_\ell ) }{ \abs{q_j-q_\ell}  } \lesssim_{M, \eta} \tau,
	\end{align*}
	where the last inequality follows from \eqref{eqTX23} in Lemma \ref{lem:cheb}. If instead $y \in \mathcal{Y}_\ell(p)$ for some $\ell \ge 2$, then for $N$ large enough depending on $M$ and $\eta$, we conclude using Lemma \ref{lem:ahlf_reg} that 
	\begin{align*}
		\frac{  \rho_p + \rho_y  }{ \abs{p-y}  } \lesssim_{M, \eta} \frac{  N r^{n-1}  }{ N^{-2}  r^{n-1-\ell} } \lesssim_{M, \eta} N^3 r^2 = (N r) \tau \ll \tau,
	\end{align*}
	for $N$ large. This completes the proof of the first claim.
	
	As for the second claim, suppose again that $p$ has index $j$ and $\hat{p} = q \in K_{n-2}$. We split the sum in \eqref{eq:geometric-sum} by $y \in \mathcal{Y}_\ell(p)$ for $1 \le \ell \le n-1$. The set of  $y \in \mathcal{Y}_1(p)$ are indexed by $y_i=q+\rho_qQ_q(q_i)$ for $i \ne j$, and for such $y_i$ we have that 
	\begin{align*}
		\abs{p - y_i} = \rho_q \abs{ Q_q(q_j) - Q_q(q_i)} \simeq \rho_q \abs{q_i - q_j},
	\end{align*}
	since for $\tau$ small, $Q_q$ is bi-Lipschitz. In particular, we see that 
	\begin{align*}
		\sum_{y \in \mathcal{Y}_1(p)} \frac{ \rho_p^{d-a} \rho_y^{a}  }{\abs{p -y }^d} \lesssim_{d, M, \eta} (Nr)^d \sin^{d-a} \theta_j \sum_{i \ne j} \frac{ \sin^a \theta_i }{ \abs{q_j - q_i}^d } \lesssim_{d, M, \eta} (N r)^d N^d \lesssim_{d, M, \eta} \tau^d 
	\end{align*}
	by virtue of Lemma \ref{lem:weighted}.
	
	On the other hand, for $y \in \mathcal{Y}_\ell(p)$ for $\ell \ge 2$ we apply Lemma \ref{lem:ahlf_reg} for $N$ large to see that $\abs{y-p} \gtrsim_{M, \eta}  N^{-2} r^{n-1-\ell}$, and since $\# \mathcal{Y}_\ell(p) \le N^\ell$, we have 
	\begin{align*}
		\sum_{\ell = 2}^{n-1} \sum_{y \in \mathcal{Y}_\ell(p)} \frac{ \rho_p^{d-a} \rho_y^a  }{ \abs{p-y}^d } \lesssim_{d, M, \eta} N^{3d} \sum_{\ell =2}^{n-1} (N r^d )^\ell   \lesssim_{d, M, \eta} N^{3d + 2} r^{2d} = N^{-d+2} \tau^{2d} \le \tau^d,
	\end{align*}
	since $d \ge 2$ and $\tau \in (0,1)$ for $N$ sufficiently large. This completes the proof.	 
\end{proof}

\subsection{Polynomial profiles of the three error terms}
\label{sec:Taylor-profiles}

We now analyze precisely the difference in $g_n(x) - g_{n-1}(\hat{x})$ (just as in Section \ref{sec:basic-construction}) provided that $K_n$ is constructed from $K_{n-1}$ via small polynomial perturbations of the Chebyshev nodes. More precisely, we assume that $K_m$ have been constructed for $0 \le m \le n-1$ as in Section \ref{sec:corrected-geometry}, and suppose that we construct $K_n$ from $K_{n-1}$ as follows.

First, suppose that for each $p \in K_{n-1}$ we have a scaling factor 
\begin{equation}
	\label{eq:current-scale-bound}
	1 \le  a_{n-1}(p)
	\lesssim_{M,\eta}N\sin\theta_{j(p)},
\end{equation}
where $j(p)$ is the index of $p$. Assume also that for each $p \in K_{n-1}$, we have a polynomial
\begin{equation}
	\label{eq:new-deformation}
	Q_p(t)
	=
	t+\sum_{\nu=1}^{M-1}P_{\nu+1,p}(t),
\end{equation}
where, for $1\le\nu\le M-1$, $P_{\nu +1 , p}$ is a polynomial satisfying 
\begin{equation}
	\label{eq:new-correction-bounds}
	\deg P_{\nu+1,p}\le \nu+1,
	\qquad
	\|P_{\nu+1,p}\|_{C^1([-1,1])}
	\lesssim_M\tau^\nu,
\end{equation}
and
\begin{equation}
	\label{eq:new-correction-normalization}
	\sum_{k=1}^N P_{\nu+1,p}(q_k)=0.
\end{equation}
Writing for convenience 
\[
\rho_p=a_{n-1}(p)r^{n-1},
\]
we construct $K_n$ by replacing each $p \in K_{n-1}$ by its $N$ children, given by the relative positions
\begin{equation}
	\label{eq:new-children}
	x_k=p+\rho_pQ_p(q_k),
	\qquad 1\le k\le N.
\end{equation}

As in Section \ref{sec:basic-construction}, for $x = x_k$ as in \eqref{eq:new-children}, we write the difference 
\begin{align*}
	g_{n}(x) - g_{n-1}(\hat{x}) = g_{n}(x_k) - g_{n-1}(p) = N^{-n} \left( \Delta_1(k) + \Delta_2(k) + \Delta_3(k) \right),
\end{align*}
where $\Delta_i(k)$ are defined exactly as in \eqref{eqn:delta}, and we emphasize the index $k$ of $x$ in our notation. In the following estimates, the polynomials
$P_{2,p},\ldots,P_{M,p}$ are still arbitrary, subject only to
\eqref{eq:new-correction-bounds} and
\eqref{eq:new-correction-normalization}.  Our purpose is to identify
the polynomial profiles in the variable $q_k$ produced by the deformations in $Q_p$ in each of the terms
\[
\Delta_1,\qquad \Delta_2,\qquad \Delta_3,
\]
by order in $\tau$.  The triangular dependence of these profiles
on the polynomials will later allow us to inductively choose $Q_p$ carefully so that there are no errors which are of low order in $\tau$; see, for example, the properties of the $A_{m,p}$ in Lemma \ref{lem:Delta1-profile}.

Finally, we remark that our analysis of the term $\Delta_2$ will require us to have some precise estimates on discrete potentials of the form
\begin{align*}
	\Lambda_{n-1, d}(p) = c_{N, p} \sum_{ y \in K_{n-1}, y \ne p } \frac{1}{(p-y)^d}.
\end{align*} 
These estimates, while not particularly difficult, are long, and so for the sake of readability, we push them to the later Section \ref{sec:higher-potential}.

\textit{In all of the lemmas of this subsection, we make the standing assumptions on the constructions of $K_1, \dotsc, K_n$, i.e., that \eqref{eq:old-children}-\eqref{eq:old-scale-bound} and \eqref{eq:current-scale-bound}-\eqref{eq:new-children} hold}.
Let us begin with the first term, $\Delta_1 = \Delta_1(k)$. Recall that for a polynomial $P$, $LP$ is defined by
\[
(LP)(q_k)
=
\sum_{\ell\ne k}
\frac{P(q_k)-P(q_\ell)}{q_k-q_\ell}.
\]

\begin{lemma}[Expansion of $\Delta_1$]
\label{lem:Delta1-profile}
As long as $N$ is chosen sufficiently large depending on $s$ and $M$, we have 
\begin{equation}\label{eq:Delta1-profile-expansion}
\Delta_1(k)
=
C_{1,p}+\beta_k
+
\sum_{m=1}^{M-1}
\left(
LP_{m+1,p}(q_k)+A_{m,p}(q_k)
\right)
+
E_{1,M,p}(q_k),
\end{equation}
where
\[
C_{1,p}=(N-1)\log\rho_p+b_*,
\]
$A_{1,p} = 0$, and $A_{m,p}\in\mathcal P_m$ is a polynomial of degree at most $m$ for $2 \le m \le M-1$. Moreover, we have the bounds
\[
\|A_{m,p}\|_{\mathrm{coeff}}
\lesssim_M N\tau^m, \qquad |E_{1,M,p}(q_k)|\lesssim_M N\tau^M,
\qquad 1\le k\le N,
\]
and each $A_{m,p}$ is determined uniquely by $P_{2,p}, \dotsc, P_{m,p}$. 
\end{lemma}

\begin{proof}
By \eqref{eq:new-children}, the definition of $x= x_k$, we see that
\[
\Delta_1(k)
=
(N-1)\log\rho_p
+
\sum_{\ell\ne k}
\log|Q_p(q_k)-Q_p(q_\ell)|.
\]
Write
\[
Q_p(q_k)-Q_p(q_\ell)
=
(q_k-q_\ell)
\left(
1+\sum_{\nu=1}^{M-1}D_{\nu,p}(q_k,q_\ell)
\right),
\]
where
\[
D_{\nu,p}(q_k,q_\ell)
=
\frac{P_{\nu+1,p}(q_k)-P_{\nu+1,p}(q_\ell)}
     {q_k-q_\ell}.
\]
Then we have
\begin{align}\label{eqDelta1k}
\Delta_1(k)
& =
(N-1)\log\rho_p
+\sum_{\ell\ne k}\log|q_k-q_\ell| + \sum_{\ell\ne k} \log\left|
1+\sum_{\nu=1}^{M-1}D_{\nu,p}(q_k,q_\ell)
\right|\nonumber\\
& = (N-1)\log\rho_p  + b_*+ \beta_k+ \sum_{\ell\ne k} \log\left|
1+\sum_{\nu=1}^{M-1}D_{\nu,p}(q_k,q_\ell)
\right|,
\end{align}
where $\beta_k$ is as in \eqref{eqn:bj}.
We now consider the last term in \eqref{eqDelta1k}. The
polynomial $D_{\nu,p}$ has total degree at most $\nu$ in
$(q_k,q_\ell)$. Moreover, by the mean value theorem and
\eqref{eq:new-correction-bounds},
\[
|D_{\nu,p}(q_k,q_\ell)|
\le
\|P_{\nu+1,p}'\|_\infty
\lesssim_M\tau^\nu.
\]
Consequently,
\[
\left|
\sum_{\nu=1}^{M-1}D_{\nu,p}(q_k,q_\ell)
\right|
\lesssim_M\tau < 1/2
\]
if $\tau$ is sufficiently small (i.e., if $N$ is sufficiently large), so we are justified in expanding the $\log$ above.

To ease notation, write $D_\nu=D_{\nu,p}(q_k,q_\ell)$. 
Then
\begin{align}
\log \abs {1+\sum_{\nu=1}^{M-1}D_\nu }
={}&
\sum_{j=1}^{M-1}
\frac{(-1)^{j+1}}{j}
\left(\sum_{\nu=1}^{M-1}D_\nu\right)^j
+
O_M(\tau^M)
\nonumber\\
={}&
\sum_{j=1}^{M-1}
\frac{(-1)^{j+1}}{j}
\sum_{\substack{\nu_1,\ldots,\nu_j=1 \\ \nu_1 + \cdots + \nu_j < M    }}^{M-1}
D_{\nu_1}\cdots D_{\nu_j}
+
O_M(\tau^M),
\label{eq:log-D-expansion}
\end{align}
where all the terms for which $
\nu_1+\cdots+\nu_j\ge M$ are included in the remainder. Grouping the remaining
terms according to $m=\nu_1+\cdots+\nu_j$ for $1 \le m \le M-1$
we obtain
\begin{equation}\label{eq:log-D-grouped}
\log\abs{1+\sum_{\nu=1}^{M-1}D_\nu}
=
\sum_{m=1}^{M-1}
\mathcal A_{m,p}(q_k,q_\ell)
+
O_M(\tau^M),
\end{equation}
where
\begin{equation}\label{eq:def-Amathcal}
\mathcal A_{m,p}(q_k,q_\ell)
=
\sum_{j=1}^{m}
\frac{(-1)^{j+1}}{j}
\sum_{\substack{\nu_1+\cdots+\nu_j=m\\
\nu_i\ge1}}
\prod_{i=1}^j
D_{\nu_i,p}(q_k,q_\ell).
\end{equation}

The term with $j=1$ in \eqref{eq:def-Amathcal} is simply
$D_{m,p}$. Hence
\begin{align}
\mathcal A_{m,p}(q_k,q_\ell)
={}&
D_{m,p}(q_k,q_\ell)
+
\sum_{j=2}^{m}
\frac{(-1)^{j+1}}{j}
\sum_{\substack{\nu_1+\cdots+\nu_j=m\\
\nu_i\ge1}}
\prod_{i=1}^j
D_{\nu_i,p}(q_k,q_\ell).
\label{eq:Amathcal-split}
\end{align}
After summing in $\ell\ne k$, the first term gives
\[
\sum_{\ell\ne k}
D_{m,p}(q_k,q_\ell)
=
LP_{m+1,p}(q_k).
\]
We define the remaining part by
\begin{equation}\label{eq:def-Amp}
A_{m,p}(q_k)
:=
\sum_{\ell\ne k}
\sum_{j=2}^{m}
\frac{(-1)^{j+1}}{j}
\sum_{\substack{\nu_1+\cdots+\nu_j=m\\
\nu_i\ge1}}
\prod_{i=1}^j
D_{\nu_i,p}(q_k,q_\ell).
\end{equation}
For $m=1$ the second sum is empty, and hence $
A_{1,p}=0$. Formula \eqref{eq:def-Amp} also makes clear why $A_{m,p}$ depends only
on the polynomials $
P_{2,p},\ldots,P_{m,p}$, but not on $P_{m+1, p}$. 
Indeed, when $j\ge2$ and $\nu_1+\cdots+\nu_j=m$, we necessarily have
\[
1\le\nu_i\le m-1.
\]
Thus $P_{m+1,p}$ occurs only in the linear term
$LP_{m+1,p}$.

We next verify the degree and size estimates for $A_{m,p}$. Since
$D_{\nu_i,p}$ has total degree at most $\nu_i$, every product in
\eqref{eq:def-Amp} has total degree at most
\[
\nu_1+\cdots+\nu_j=m
\]
in $(q_k,q_\ell)$. Expanding $A_{m,p}$ into monomials, summing in $\ell$, and using the identities
\[
\sum_{\ell\ne k}q_\ell^a
=
S_a-q_k^a,
\qquad
S_a=\sum_{\ell=1}^Nq_\ell^a,
\]
we see that $A_{m,p}$ is a polynomial in $q_k$ of degree at most
$m$. Furthermore, 
$|D_{\nu_i,p}|
\lesssim_M\tau^{\nu_i}$, and
so every product in \eqref{eq:def-Amp} is bounded by
\[
C_M\tau^{\nu_1+\cdots+\nu_j}
=
C_M\tau^m.
\]
Since, for fixed $M$, there are only $O_M(1)$ possible compositions
of $m$, and the outer sum in \eqref{eq:def-Amp} contains at most $N$ terms, we obtain
\[
\|A_{m,p}\|_{\mathrm{coeff}}
\lesssim_M
N\tau^m.
\]
Since all the polynomials involved have degree bounded in terms of \(M\), their coefficient norms are controlled by their uniform norms, with constants depending only on \(M\).

Finally, the terms omitted from \eqref{eq:log-D-grouped} contain a
total power at least $\tau^M$. For each fixed pair $(k,\ell)$ their
sum is $O_M(\tau^M)$, and hence, after summing over $\ell\ne k$,
\[
|E_{1,M,p}(q_k)|
\lesssim_M
N\tau^M,
\]
which completes the proof.
\end{proof}

We now move to the terms $\Delta_2$ and $\Delta_3$, but first we need some notation. For $d\ge1$ define
\begin{align}\label{eq:Fnd-new}
F_{n-1,d}(p)
& =
\sum_{\substack{y\in K_{n-1} \setminus \{p\}}}
\frac1{(p-y)^d},  \\
\Lambda_{n-1,d}(p)
& =
N\rho_p^dF_{n-1,d}(p).
\end{align}
The factor $N$ is included because, up to the factor
$(-1)^{d-1}/d$, $\Lambda_{n-1,d}(p)$ is precisely the $d$-th coefficient which
appears in the Taylor expansion of $\Delta_2$. For $\Delta_2$, we have:

\begin{lemma}[Expansion of $\Delta_2$]
\label{lem:Delta2-profile}
Assume that for all $p \in K_{n-1}$,
\begin{equation}\label{eq:field-hypothesis-profile}
|\Lambda_{n-1,d}(p)|
\le C_MN\tau^d,
\qquad 1\le d\le M.
\end{equation}
As long as $N$ is large enough, depending on $M$, $s$, and $\eta$,
\begin{equation}\label{eq:Delta2-profile-expansion}
\Delta_2(k)
=
\sum_{m=1}^{M-1}B_{m,p}(q_k)
+
E_{2,M,p}(q_k),
\end{equation}
where $B_{m,p}\in\mathcal P_m$ is a polynomial of degree at most $m$, and we have the bounds 
\[
\|B_{m,p}\|_{\mathrm{coeff}}
\lesssim_M N\tau^m, \qquad |E_{2,M,p}(q_k)|\lesssim_{M, \eta} N\tau^M,
\qquad 1\le k\le N.
\]
Moreover, each polynomial $B_{m,p}$ is determined uniquely by the polynomials $P_{2,p}, \dotsc P_{m,p}$ and depends 
only. More precisely, $B_{m,p}$ depends only on the exterior field
coefficients $\Lambda_{n-1,d}(p)$ and on the correction
polynomials $P_{\nu,p}$ with $\nu\le m$.
In particular, it is independent of $P_{m+1,p}$.

\end{lemma}
\begin{proof}
Define the small parameter $h_k=\rho_pQ_p(q_k)$, and set
\[
    H_p(h)
    =
    N\sum_{\substack{y\in K_{n-1}\\y\ne p}}
    \left(
        \log|p+h-y|-\log|p-y|
    \right),
\]
so that $\Delta_2(k)=H_p(h_k)$. We first justify the Taylor expansion of $H_p$ at the point $h_k$.
By \eqref{eq:new-correction-bounds},
\[
    \|Q_p-\operatorname{Id}\|_{C^1([-1,1])}
    \lesssim_M\tau,
\]
and hence for $\tau$ small enough, 
\[
    |Q_p(q_k)|\lesssim_M1,
    \qquad
    |h_k|\lesssim_M\rho_p.
\]
Estimate \eqref{eq:relative-scale} of Lemma~\ref{lem:prev_estimates} therefore gives
\[
    \left|\frac{h_k}{p-y}\right|
    \lesssim_{M,\eta}\tau,
    \qquad
    y\in K_{n-1},\quad y\ne p.
\]
Taking $\tau$ sufficiently small, we may assume that
\[
    \left|\frac{h_k}{p-y}\right|\le\frac12
\]
uniformly in $y$ and $k$.

Thus, for every $y\ne p$,
\[
    \log\left|1+\frac{h_k}{p-y}\right|
    =
    \sum_{d=1}^{M}
    \frac{(-1)^{d-1}}{d}
    \left(\frac{h_k}{p-y}\right)^d
    +
    R_M(h_k,p-y),
\]
where the remainder satisfies $|R_M(h_k,p-y)|
\lesssim_M
\left|h_k/(p-y)\right|^{M+1}$. 
After summing over $y\in K_{n-1}\setminus\{p\}$, we obtain
\begin{equation}
\label{eq:Delta2-Taylor}
    \Delta_2(k)
    =
    \sum_{d=1}^{M}
    \frac{(-1)^{d-1}}{d}
    \Lambda_{n-1,d}(p)Q_p(q_k)^d
    +
    \mathcal R_{M,p}(h_k),
\end{equation}
by the definition of $h_k$ and $\Lambda_{n-1,d}$. The Taylor remainder satisfies 
\begin{align}
    |\mathcal R_{M,p}(h_k)|
    &\lesssim_M
    N|h_k|^{M+1}
    \sum_{\substack{y\in K_{n-1}\\y\ne p}}
    \frac{1}{|p-y|^{M+1}}
    \nonumber\\
    &\lesssim_M
    N\rho_p^{M+1}
    \sum_{\substack{y\in K_{n-1}\\y\ne p}}
    \frac1{|p-y|^{M+1}} \lesssim_M
    N\tau^{M+1}
\label{eq:Hp-Taylor-remainder}
\end{align}
where the last inequality follows from Lemma~\ref{lem:prev_estimates} with $d=M+1$ and $a=0$.

By \eqref{eq:field-hypothesis-profile},
\[
|\Lambda_{n-1,d}(p)|
\lesssim_M N\tau^d.
\]

It remains to analyze the explicit terms in
\eqref{eq:Delta2-Taylor}. For each $1 \le d \le M $, write
\[
    Q_p(q_k)^d
    =
    \left(
        q_k+\sum_{\nu=1}^{M-1}P_{\nu+1,p}(q_k)
    \right)^d.
\]
Expanding the product yields a sum of terms of the form
\begin{equation}
\label{eq:Qd-typical-term}
    c(d;\nu_1,\ldots,\nu_j)\,
    q_k^{d-j}
    P_{\nu_1+1,p}(q_k)\cdots
    P_{\nu_j+1,p}(q_k),
\end{equation}
where
\[
    0\le j\le d,
    \qquad
    1\le\nu_1,\ldots,\nu_j\le M-1.
\]
The case $j=0$ is understood to give simply $q_k^d$. By the assumption \eqref{eq:new-correction-bounds},
\[
    \left|
    P_{\nu_1+1,p}(q_k)\cdots
    P_{\nu_j+1,p}(q_k)
    \right|
    \lesssim_M
    \tau^{\nu_1+\cdots+\nu_j}.
\]
Consequently, after multiplication by
$\Lambda_{n-1,d}(p)$, the term
\eqref{eq:Qd-typical-term} has size
\begin{equation}
\label{eq:Delta2-typical-size}
    O_M\left(
        N\tau^{d+\nu_1+\cdots+\nu_j}
    \right)
\end{equation}
because we assume \eqref{eq:field-hypothesis-profile}. Moreover, its degree in $q_k$ is at most
\[
    (d-j)+\sum_{i=1}^j(\nu_i+1)
    =
    d+\sum_{i=1}^j\nu_i.
\]
Thus the same quantity
\[
    m=d+\nu_1+\cdots+\nu_j
\]
controls both the order in $\tau$ and the polynomial degree.

For $1\le m\le M-1$, let $B_{m,p}$ be the sum of all the terms as in \eqref{eq:Qd-typical-term}
for which $d+\nu_1+\cdots+\nu_j=m$, i.e., 
\begin{equation}
\label{eq:def-Bmp}
B_{m,p}(q_k)
=
\sum_{d=1}^{m}
\frac{(-1)^{d-1}}{d}\Lambda_{n-1,d}(p)
\sum_{j=0}^{d}
\sum_{\substack{
\nu_1,\ldots,\nu_j\ge1\\
d+\nu_1+\cdots+\nu_j=m
}}
c(d;\nu_1,\ldots,\nu_j)\,
q_k^{d-j}
\prod_{i=1}^jP_{\nu_i+1,p}(q_k),
\end{equation}
with the usual convention when $j=0$. In particular, the term corresponding to $d = m$ in the above is 
\[
    \frac{(-1)^{m-1}}{m}
    \Lambda_{n-1,m}(p)q_k^m.
\]
The preceding degree estimate showed that $B_{m,p}\in\mathcal P_m$ is a polynomial of degree at most $m$. Also, each summand in \eqref{eq:def-Bmp} is bounded by $C_MN\tau^m$. 
Since for fixed $M$ there are only $O_M(1)$ possible choices of
$d,j,\nu_1,\ldots,\nu_j$, we obtain our desired bound on $B_{m,p}$:
\[
    \|B_{m,p}\|_{\mathrm{coeff}}
    \lesssim_M N\tau^m.
\]
Moreover, by the definition of $B_{m,p}(q_k)$ as in \eqref{eq:def-Bmp}, since the sum starts at $d =1$, we see $B_{m,p}$ is a linear combination of product of polynomials $P_{\nu+1, p}$ with $\nu+1 \le m$.

Finally, all explicit terms in \eqref{eq:Delta2-Taylor} whose total
order
\[
    d+\nu_1+\cdots+\nu_j
\]
is at least $M$ contribute $O_M(N\tau^M)$, since $\tau<1$.
Combining these terms with the Taylor remainder gives
\[
    \Delta_2(k)
    =
    \sum_{m=1}^{M-1}B_{m,p}(q_k)
    +
    E_{2,M,p}(q_k),
\]
with $|E_{2,M,p}(q_k)|
\lesssim_M N\tau^M$, proving the lemma.
\end{proof}

Our final estimate of this section is on $\Delta_3$: 

\begin{lemma}[Expansion of $\Delta_3$]
\label{lem:cluster-profiles-new} For $N$ large depending on $M$, $s$, and $\eta$, we have 
\begin{equation}
\label{eq:Delta3-profile-expansion}
\Delta_3(k)
=
\sum_{m=2}^{M-1}C_{m,p}^{(3)}(q_k)
+
E_{3,M,p}(q_k),
\end{equation}
where $C_{m,p}^{(3)}\in\mathcal P_m$ is a polynomial of degree at most $m$, and we have the bounds
\[
\|C_{m,p}^{(3)}\|_{\mathrm{coeff}}
\lesssim_MN\tau^m, \qquad |E_{3,M,p}(q_k)|
\lesssim_{M, \eta }N\tau^M,
\qquad
1\le k\le N.
\]
Moreover, the polynomial $C_{2,p}^{(3)}$ is constant, and each $C_{m,p}^{(3)}$ is determined uniquely by the collection of polynomials
\[
P_{\nu+1,z},
\qquad
z\in K_{n-1},
\qquad
1\le \nu\le m-2.
\]

\end{lemma}

\begin{proof}
Let $y\in K_{n-1}\setminus\{p\}$ and write its children as
\[
    z_\ell
    =
    y+u_\ell,
    \qquad
    u_\ell
    =
    \rho_yQ_y(q_\ell),
    \qquad
    1\le\ell\le N.
\]
By \eqref{eq:new-correction-normalization} and
$\sum_{\ell=1}^Nq_\ell=0$, $\sum_{\ell=1}^NQ_y(q_\ell)=0$, and therefore 
\begin{equation}
\label{eq:cluster-barycenter-cancellation}
    \sum_{\ell=1}^Nu_\ell=0.
\end{equation}

Recall that $x = x_k$ is a child of $p$ with index $k$, so 
\[
    x_k
    =
    p+h_k,
    \qquad
    h_k
    =
    \rho_pQ_p(q_k).
\]
By \eqref{eq:new-correction-bounds},
\[
    |Q_p(q_k)|+|Q_y(q_\ell)|
    \lesssim_M1,
\]
and hence
\[
    |h_k|\lesssim_M\rho_p,
    \qquad
    |u_\ell|\lesssim_M\rho_y.
\]
Lemma~\ref{lem:prev_estimates} then gives (for $N$ sufficiently large)
\[
    |h_k|
    \lesssim_{M,\eta}
    \tau|p-y|,
    \qquad
    |u_\ell|
    \lesssim_{M,\eta}
    \tau|p-y|.
\]
Thus, for $\tau$ sufficiently small, $|x_k-y|
\simeq
|p-y|$ and $\abs{u_\ell}/\abs{x_k-y} \le 1/2$, uniformly in all the parameters under consideration.

We may therefore expand the logarithm around $x_k-y$ up to
order $M-1$:
\begin{align*}
    &\sum_{\ell=1}^N
    \left(
        \log|x_k-y-u_\ell|
        -
        \log|x_k-y|
    \right)  =
    \sum_{\ell=1}^N
    \log\left|1-\frac{u_\ell}{x_k-y}\right|=
    -
    \sum_{a=2}^{M-1}
    \frac1a
    \frac{\sum_{\ell=1}^N u_\ell^a}
         {(x_k-y)^a}
    +\mathcal R_y,
\end{align*}
where the linear term in $u_\ell$ (corresponding to $a=1$)
vanishes by \eqref{eq:cluster-barycenter-cancellation}.
Since
\[
    \left|\frac{u_\ell}{x_k-y}\right|
    \lesssim_{M,\eta}\tau
\]
and $|x_k-y|\approx |p-y|$, the remainder satisfies
\[
\begin{aligned}
    |\mathcal R_y|
    &\lesssim_M
    \sum_{\ell=1}^N
    \left|\frac{u_\ell}{x_k-y}\right|^M\lesssim_{M,\eta}
    N\frac{\rho_y^M}{|p-y|^M}.
\end{aligned}
\]
Consequently, by Lemma~\ref{lem:prev_estimates}
with $d=a=M$, we obtain
\[
\begin{aligned}
    \sum_{\substack{y\in K_{n-1}\\y\ne p}}
    |\mathcal R_y|
    &\lesssim_{M,\eta}
    N\sum_{\substack{y\in K_{n-1}\\y\ne p}}
    \frac{\rho_y^M}{|p-y|^M}\lesssim_{M,\eta}N\tau^M.
\end{aligned}
\]
Thus the total contribution of these remainders
is of the required order, uniformly in the generation.

We next expand
\[
    (x_k-y)^{-a}
    =
    (p-y+h_k)^{-a}
\]
around $p-y$. Up to terms of total order at least $M$,
this produces terms of the form
\begin{equation}
\label{eq:cluster-general-term}
    c_{a,b}\,
    \frac{
        \rho_y^a\rho_p^b
    }{
        (p-y)^{a+b}
    }
    \left(
        \sum_{\ell=1}^N Q_y(q_\ell)^a
    \right)
    Q_p(q_k)^b,
\end{equation}
where
\[
    a\ge2,
    \qquad b\ge0,
\]
and $c_{a,b}$ depends only on $a$ and $b$.
The case $b=0$ is independent of $q_k$ and will be
retained as the constant part of the corresponding profile.

We now expand the polynomial perturbations in both $Q_y^a$ and
$Q_p^b$. A typical monomial in
\[
    Q_y(q_\ell)^aQ_p(q_k)^b
\]
has the form
\begin{equation}
\label{eq:cluster-corrected-monomial}
    q_\ell^{a-r}
    q_k^{b-s}
    \prod_{i=1}^r
        P_{\nu_i+1,y}(q_\ell)
    \prod_{j=1}^s
        P_{\mu_j+1,p}(q_k),
\end{equation}
where
\[
    0\le r\le a,
    \qquad
    0\le s\le b, \qquad \text{ and } 1\le\nu_i,\mu_j\le M-1.
\]
Set
\[
    \sigma
    =
    \sum_{i=1}^r\nu_i
    +
    \sum_{j=1}^s\mu_j.
\]
Since \eqref{eq:new-correction-bounds} holds both at $y$ and at $p$,
the product of the polynomials in
\eqref{eq:cluster-corrected-monomial} is bounded by
\[
    C_M\tau^\sigma.
\]

We claim that the contribution obtained from
\eqref{eq:cluster-general-term} and
\eqref{eq:cluster-corrected-monomial}, after summing over
$y\in K_{n-1}\setminus\{p\}$, is bounded by
\begin{equation}
\label{eq:cluster-general-bound}
    C_MN\tau^{a+b+\sigma}.
\end{equation}
Indeed, the sum over the $N$ children of $y$ contributes at most a
factor $C_MN$, and hence the absolute value of the corresponding term
is bounded by
\[
    C_MN\tau^\sigma
    \sum_{\substack{y\in K_{n-1}\\y\ne p}}
    \frac{\rho_y^a\rho_p^b}
         {|p-y|^{a+b}}.
\]
Applying Lemma~\ref{lem:prev_estimates} with exponent $a+b$ and with
the power of $\rho_y$ equal to $a$, we obtain
\[
    \sum_{\substack{y\in K_{n-1}\\y\ne p}}
    \frac{\rho_y^a\rho_p^b}
         {|p-y|^{a+b}}
    \lesssim_{M, \eta}
    \tau^{a+b},
\]
since $a\ge2$. This proves
\eqref{eq:cluster-general-bound}.

We now organize the expansion according to its total order. Set
\begin{equation}
\label{eq:cluster-total-order}
    m
    =
    a+b+\sigma.
\end{equation}
By \eqref{eq:cluster-general-bound}, every term of total order $m$ is
bounded by $C_MN\tau^m$. Moreover, its degree in the child variable $q_k$ is also at most $m$. Indeed, the
only factor depending on $q_k$ in
\eqref{eq:cluster-corrected-monomial} is
\[
    q_k^{b-s}
    \prod_{j=1}^sP_{\mu_j+1,p}(q_k),
\]
whose degree is at most
\[
    (b-s)
    +
    \sum_{j=1}^s(\mu_j+1)
    =
    b+\sum_{j=1}^s\mu_j
    \le
    a+b+\sigma
    =
    m.
\]
Thus all terms of total order $m$ form a polynomial in $q_k$ of
degree at most $m$.

The same bookkeeping gives the triangular dependence needed later.
Since $a\ge2$, from
\[
    m=a+b+\sigma
\]
we have terms of the form \eqref{eq:cluster-corrected-monomial} must satisfy $\nu_i, \mu_j \le \sigma \le m-2$. Consequently the terms of total order $m$ depend only on
\[
    P_{\nu+1,z},
    \qquad
    z\in K_{n-1},
    \qquad
    1\le\nu\le m-2.
\]

For $2\le m\le M-1$, let $C_{m,p}^{(3)}$ be the sum of all explicit
terms for which
\[
    a+b+\sigma=m.
\]
Then $C_{m,p}^{(3)}
\in
\mathcal P_m$ and 
\[
    \|C_{m,p}^{(3)}\|_{\mathrm{coeff}}
    \lesssim_M
    N\tau^m.
\]
The terms with $b=0$ are independent of $q_k$ and are included as the
constant terms of $C_{m,p}^{(3)}$. A nonconstant term must have $b\ge1$. Since $a\ge2$, every
nonconstant term has total order
\[
    m=a+b+\sigma\ge3.
\]
Thus $C_{2,p}^{(3)}$ is constant, and the first possible non-constant
cluster polynomial occurs with $m \ge 3$, which corresponds to $\tau^3$.

It remains to estimate the discarded terms. All explicit terms for which the total order satisfies
\[
    a+b+\sigma\ge M
\]
are bounded, after summation, by
\[
    C_MN\tau^M
\]
using \eqref{eq:cluster-general-bound} and $\tau<1$. The Taylor remainder arising from the expansion of
$(p-y+h_k)^{-a}$ is estimated using Lemma \ref{lem:prev_estimates}. Namely, every such term
contains a factor of the form
\[
\frac{\rho_y^a\rho_p^b}
{|p-y|^{a+b}}
\]
with $a+b\ge M$, and Lemma~\ref{lem:prev_estimates} gives again an
$O_M(N\tau^M)$ contribution. Collecting all discarded terms into $E_{3,M,p}$ (including the first Taylor remainder from the first expansion of the $\log$) we conclude that
\[
    |E_{3,M,p}(q_k)|
    \lesssim_M
    N\tau^M,
    \qquad
    1\le k\le N.
\]
This proves \eqref{eq:Delta3-profile-expansion}.
\end{proof}

Combining Lemmas \ref{lem:Delta1-profile}-\ref{lem:cluster-profiles-new}, our main result of this section is the following decomposition.

\begin{lemma}[Expansion of the errors]
\label{lem:higher-basic-error}
Fix $M\ge2$, $\eta > 0$, and suppose that $K_m$ have been constructed for $0 \le m \le n$ so as to satisfy the standing assumptions \eqref{eq:old-children}-\eqref{eq:old-scale-bound} and \eqref{eq:current-scale-bound}-\eqref{eq:new-children}. Assume in addition that we have the bound
\[
|\Lambda_{n-1,d}(p)|
\lesssim_MN\tau^d,
\qquad
1\le d\le M,
\]
uniformly for $p\in K_{n-1}$. Then for $N$ chosen large enough, depending on $M$, $s$, and $\eta$, the following holds. 

For $x = x_k = p + \rho_p Q_p(q_k) \in K_n$ and $g_{n}(x) - g_{n-1}(p) = N^{-n} \left( \Delta_1(k) + \Delta_{2}(k) + \Delta_3(k) \right)$ defined as usual, we have 
\begin{align}
\Delta_1(k)+\Delta_2(k)+\Delta_3(k)
={}&(N-1)\log\rho_p+b_*+\beta_k \nonumber\\
&+\sum_{m=1}^{M-1}
\bigl(LP_{m+1,p}(q_k)+R_{m,p}(q_k)\bigr)
+E_{M,p}(q_k),
\label{eq:combined-profile-expansion}
\end{align}
where $R_{m,p}\in\mathcal P_m$ is a polynomial of degree at most $m$, and
\[
\|R_{m,p}\|_{\mathrm{coeff}}\lesssim_MN\tau^m, \qquad |E_{M,p}(q_k)|
\lesssim_{M, \eta } N\tau^M,
\qquad 1\le k\le N.
\]
Moreover, the polynomial $R_{m,p}$ is determined uniquely by the polynomials 
\[
    P_{\nu+1,z},
    \qquad
    z\in K_{n-1},\quad 1\le \nu\le m-1.
\]
In particular, it is independent of all corrections
$P_{m+1,z}$.
\end{lemma}

\begin{proof}
Lemma~\ref{lem:Delta1-profile} gives the basic scalar term
$(N-1)\log\rho_p+b_*$, the fixed profile $\beta_k$, and the
perturbative profiles $LP_{m+1,p}+A_{m,p}$. Lemma~\ref{lem:Delta2-profile} gives
\[
\Delta_2(k)=\sum_{m=1}^{M-1}B_{m,p}(q_k)+E_{2,M,p}(q_k),
\]
and similarly,
Lemma~\ref{lem:cluster-profiles-new} gives
\[
\Delta_3(k)=\sum_{m=2}^{M-1}C_{m,p}^{(3)}(q_k)
+E_{3,M,p}(q_k).
\]
Put
\[
C_{1,p}^{(3)}=0,\qquad
R_{m,p}=A_{m,p}+B_{m,p}+C_{m,p}^{(3)}.
\]
The degree, size, and triangular-dependence assertions follow from
the preceding three lemmas. Summing the pointwise remainder estimates gives
\[
|E_{M,p}(q_k)|\lesssim_MN\tau^M,
\]
as desired.
\end{proof}

\section{Discrete potential estimates}
\label{sec:higher-potential}
The estimate for the external term in
Lemma~\ref{lem:higher-basic-error} was proved under the hypothesis
\[
    |\Lambda_{n-1,d}(p)|
    \lesssim_{M,\eta} N\tau^d,
    \qquad
    1\le d\le M,
    \qquad p\in K_{n-1}.
\]
We now prove the corresponding estimate uniformly at every generation,  assuming only that the scalars $a_{m}(z)$ satisfy the bounds $1 \le a_{m}(z) \lesssim_{M,\eta} N \sin \theta_{j(z)}$, and that the deformations $Q_z$ are close to the identity; see Lemma \ref{lem:higher-potential-new}.

For this purpose, we use $n$ to denote the current generation.
Thus, for $p\in K_n$, write
\[
    \rho_p=a_n(p)r^n,
\]
and define
\[
    F_{n,d}(p)
    =
    \sum_{\substack{y\in K_n\\y\ne p}}
    \frac1{(p-y)^d},
    \qquad
    \Lambda_{n,d}(p)
    =
    N\rho_p^dF_{n,d}(p).
\]
Throughout this section we assume, for all generations under
consideration, the geometric hypotheses
\eqref{eq:old-scale-bound}, \eqref{eq:Q-close-id}, and
\eqref{eq:Q-barycenter} from
Section~\ref{sec:corrected-geometry}.  In particular,
\eqref{eq:Q-bilip} follows from \eqref{eq:Q-close-id}.
The estimates below depend only on these uniform properties, and not
on the particular choice of the deformations $Q_z$. Let $p\in K_n$ and let $q=\widehat p\in K_{n-1}$ be its parent. Then we have the following decomposition of $F_{n,d}(p)$.

\begin{lemma}
\label{lem:field-decomposition}
For every $d\ge1$,
\begin{equation}
\label{eq:field-recursion-new}
F_{n,d}(p)
=
S_{n,d}(p)
+
NF_{n-1,d}(q)
+
E_{n,d}^{\mathrm{ev}}(p)
+
E_{n,d}^{\mathrm{cl}}(p),
\end{equation}
where
\[
S_{n,d}(p)
=
\sum_{\substack{y\in K_n\\ \widehat y=q,\ y\ne p}}
(p-y)^{-d},
\]
\[
E_{n,d}^{\mathrm{ev}}(p)
=
N\sum_{\substack{z\in K_{n-1}\\z\ne q}}
\left[
(p-z)^{-d}-(q-z)^{-d}
\right],
\]
and
\[
E_{n,d}^{\mathrm{cl}}(p)
=
\sum_{\substack{z\in K_{n-1}\\z\ne q}}
\left[
\sum_{\widehat y=z}(p-y)^{-d}
-
N(p-z)^{-d}
\right].
\]
For $n=1$ we also have $F_{1,d}(p)=S_{1,d}(p)$.
\end{lemma}

\begin{proof}
Group the points $y\in K_n\setminus\{p\}$ according to their parents.
The points with parent $q$ give $S_{n,d}(p)$.  For every other parent
$z$, add and subtract $N(p-z)^{-d}$, and then add and subtract
$N(q-z)^{-d}$.  This gives \eqref{eq:field-recursion-new}.
\end{proof}

For later use, set
\begin{equation}
\label{eq:Xnd-definition}
    X_{n,d}
    =
    \sup_{p\in K_n}
    \frac{|\Lambda_{n,d}(p)|}{N\tau^d},
    \qquad
    X_n
    =
    \max_{1\le d\le M}X_{n,d}.
\end{equation}
Since $K_0=\{0\}$, we may put $X_{0,d}=0$. The four terms in \eqref{eq:field-recursion-new} play different roles.
The sibling term $S_{n,d}(p)$ can be bounded directly.  The transmission term $N F_{n-1,d}(q)$ carries
the field from the previous generation and gives the only contribution
of size $O(\tau X_{n-1})$.  The evaluation and cluster errors, $E_{n,d}^{\mathrm{ev}}(p)$ and $E_{n,d}^{\mathrm{cl}}(p)$ respectively,  are of
higher order in $\tau$. We now estimate each of these terms separately.

\begin{lemma}
\label{lem:field-local}
For $N$ sufficiently large, the following estimates hold for every
$1\le d\le M$ and $n\ge1$:
\begin{equation}
\label{eq:sibling-field-new}
N\rho_p^d|S_{n,d}(p)|
\lesssim_{M,\eta}
N\tau^d.
\end{equation}
If $n\ge2$, then
\begin{equation}
\label{eq:field-transmission-normalized}
\frac{
N\rho_p^d\,N|F_{n-1,d}(q)|
}{
N\tau^d
}
\lesssim_{M,\eta}
\tau X_{n-1,d}.
\end{equation}
\end{lemma}

\begin{proof}
Suppose that $p$ is the $j$-th child of $q$.  Since the children of
$q$ were constructed using $Q_q$,
\[
S_{n,d}(p)
=
\rho_q^{-d}
\sum_{\ell\ne j}
\frac1{(Q_q(q_j)-Q_q(q_\ell))^d}.
\]
The scale bound gives
\begin{equation}
\label{eq:rho-child-parent}
\frac{\rho_p}{\rho_q}
=
r\frac{a_n(p)}{a_{n-1}(q)}
\lesssim_{M,\eta}
Nr\sin\theta_j.
\end{equation}

We first consider $d\ge2$. Since $Q_q$ is bi-Lipschitz as in \eqref{eq:Q-bilip}, we apply Lemma \ref{lem:weighted} with $a = 0$ to see that 
\[
\sum_{\ell\ne j}
\frac1{|Q_q(q_j)-Q_q(q_\ell)|^d}
\lesssim_{d,M}
\frac{N^d}{\sin^d\theta_j}.
\]
Hence we conclude 
\begin{align*}
N\rho_p^d|S_{n,d}(p)|
&\lesssim_{M,\eta}
N(Nr)^d\sin^d\theta_j
\frac{N^d}{\sin^d\theta_j} = C_{M,\eta}N\tau^d.
\end{align*}

It remains to treat $d=1$, where the absolute-value estimate above
would lose a logarithm.  By \eqref{eq:Q-close-id},
\[
Q_q(q_j)-Q_q(q_\ell)
=
(q_j-q_\ell)(1+\varepsilon_{j\ell}),
\qquad
|\varepsilon_{j\ell}|\lesssim_M\tau.
\]
For $\tau$ sufficiently small,
\[
\frac1{Q_q(q_j)-Q_q(q_\ell)}
=
\frac1{q_j-q_\ell}
+
O_M\left(\frac{\tau}{|q_j-q_\ell|}\right).
\]
Using \eqref{eqTX21} and the Chebyshev gap estimate from the proof of
Lemma~\ref{lem:weighted},
\[
\left|\sum_{\ell\ne j}\frac1{q_j-q_\ell}\right|
\lesssim
\frac1{\sin^2\theta_j},
\qquad
\sum_{\ell\ne j}\frac1{|q_j-q_\ell|}
\lesssim
\frac{N\log N}{\sin\theta_j}.
\]
Consequently,
\[
|S_{n,1}(p)|
\lesssim_M
\rho_q^{-1}
\left(
\frac1{\sin^2\theta_j}
+
\frac{\tau N\log N}{\sin\theta_j}
\right).
\]
Combining this with \eqref{eq:rho-child-parent} and
$\sin\theta_j\gtrsim N^{-1}$ gives
\[
N\rho_p|S_{n,1}(p)|
\lesssim_{M,\eta}
N\tau+N\tau^2\log N.
\]
Since $s<1/2$, we have $\tau\log N\to0$ as $N\to\infty$; after
increasing $N$ if necessary, the last expression is
$O_{M,\eta}(N\tau)$.  This proves
\eqref{eq:sibling-field-new} also for $d=1$.

For the transmission term,
\[
N\rho_p^d\,N|F_{n-1,d}(q)|
=
N
\left(\frac{\rho_p}{\rho_q}\right)^d
|\Lambda_{n-1,d}(q)|.
\]
Using
\[
\frac{\rho_p}{\rho_q}
\lesssim_{M,\eta}
Nr
=
\frac{\tau}{N}
\]
and the definition of $X_{n-1,d}$,
\begin{align*}
\frac{
N\rho_p^d\,N|F_{n-1,d}(q)|
}{
N\tau^d
}
&\lesssim_{M,\eta}
N^{1-d}\tau^dX_{n-1,d}\\
&\le
C_{M,\eta}\tau X_{n-1,d},
\end{align*}
because $d\ge1$ and $0<\tau<1$.
\end{proof}

\begin{lemma}
\label{lem:field-evaluation}
Fix $1\le d\le M$ and $n\ge2$, and write
\[
    p=q+h.
\]
Then
\begin{equation}
\label{eq:Eev-expansion-new}
E_{n,d}^{\mathrm{ev}}(p)
=
N\sum_{t=1}^{M-d}
c_{d,t}h^tF_{n-1,d+t}(q)
+
\mathcal R_{n,d}^{\mathrm{ev}}(p),
\end{equation}
where
\[
c_{d,t}
=
(-1)^t\binom{d+t-1}{t},
\]
with the convention that the sum is empty when $d=M$.

Moreover, for $1\le t\le M-d$,
\begin{equation}
\label{eq:Eev-explicit-new}
\frac{
N\rho_p^d
\left|Nh^tF_{n-1,d+t}(q)\right|
}{
N\tau^d
}
\lesssim_{M,\eta}
\tau^{t+1}X_{n-1,d+t},
\end{equation}
and
\begin{equation}
\label{eq:Eev-rem-new}
\frac{
N\rho_p^d
|\mathcal R_{n,d}^{\mathrm{ev}}(p)|
}{
N\tau^d
}
\lesssim_{M,\eta}
\tau^{M+1}.
\end{equation}
\end{lemma}

\begin{proof}
If $p$ is the $j$-th child of $q$, then
\[
    h=\rho_qQ_q(q_j),
\]
and hence $|h|\lesssim_M\rho_q$.  For
$z\in K_{n-1}\setminus\{q\}$, Lemma~\ref{lem:prev_estimates} gives
\[
    \frac{|h|}{|q-z|}
    \lesssim_{M,\eta}
    \frac{\rho_q}{|q-z|}
    \lesssim_{M,\eta}\tau.
\]
After taking $\tau$ sufficiently small, this ratio is at most $1/2$.
Thus the Taylor expansion below is uniform in $z$.

Taylor's formula gives
\[
(p-z)^{-d}
=
\sum_{t=0}^{M-d}
c_{d,t}
\frac{h^t}{(q-z)^{d+t}}
+
O_M\left(
\frac{|h|^{M-d+1}}{|q-z|^{M+1}}
\right).
\]
Subtracting the term with $t=0$ and summing over $z\ne q$ gives
\eqref{eq:Eev-expansion-new}.

For an explicit term,
\[
N\rho_p^d
\left|Nh^tF_{n-1,d+t}(q)\right|
=
N
\left(\frac{\rho_p}{\rho_q}\right)^d
\left(\frac{|h|}{\rho_q}\right)^t
|\Lambda_{n-1,d+t}(q)|.
\]
Since
\[
    |h|\lesssim_M\rho_q,
    \qquad
    \frac{\rho_p}{\rho_q}
    \lesssim_{M,\eta}\frac{\tau}{N},
\]
the definition of $X_{n-1,d+t}$ yields
\begin{align*}
\frac{
N\rho_p^d
\left|Nh^tF_{n-1,d+t}(q)\right|
}{
N\tau^d
}
&\lesssim_{M,\eta}
N^{1-d}\tau^{d+t}X_{n-1,d+t}\\
&\le
C_{M,\eta}\tau^{t+1}X_{n-1,d+t}.
\end{align*}
This proves \eqref{eq:Eev-explicit-new}.

For the remainder,
\begin{align*}
N\rho_p^d
|\mathcal R_{n,d}^{\mathrm{ev}}(p)|
&\lesssim_M
N^2\rho_p^d
|h|^{M-d+1}
\sum_{\substack{z\in K_{n-1}\\z\ne q}}
\frac1{|q-z|^{M+1}}\\
&\lesssim_{M,\eta}
N^2
\left(\frac{\rho_p}{\rho_q}\right)^d
\rho_q^{M+1}
\sum_{\substack{z\in K_{n-1}\\z\ne q}}
\frac1{|q-z|^{M+1}}.
\end{align*}
Applying Lemma~\ref{lem:prev_estimates} to $q$ with exponent
$M+1$ and $a=0$,
\[
\rho_q^{M+1}
\sum_{\substack{z\in K_{n-1}\\z\ne q}}
\frac1{|q-z|^{M+1}}
\lesssim_{M,\eta}
\tau^{M+1}.
\]
Hence
\[
N\rho_p^d
|\mathcal R_{n,d}^{\mathrm{ev}}(p)|
\lesssim_{M,\eta}
N^{2-d}\tau^{d+M+1}.
\]
After division by $N\tau^d$,
\[
\frac{
N\rho_p^d
|\mathcal R_{n,d}^{\mathrm{ev}}(p)|
}{
N\tau^d
}
\lesssim_{M,\eta}
N^{1-d}\tau^{M+1}
\le
C_{M,\eta}\tau^{M+1}.
\]
\end{proof}

\begin{lemma}
\label{lem:field-cluster}
For $1\le d\le M$ and $n\ge2$,
\begin{equation}
\label{eq:Ecl-field-new}
N\rho_p^d
|E_{n,d}^{\mathrm{cl}}(p)|
\lesssim_{M,\eta}
N\tau^{d+2}.
\end{equation}
Equivalently,
\begin{equation}
\label{eq:Ecl-field-normalized}
\frac{
N\rho_p^d
|E_{n,d}^{\mathrm{cl}}(p)|
}{
N\tau^d
}
\lesssim_{M,\eta}
\tau^2.
\end{equation}
\end{lemma}

\begin{proof}
For $z\ne q$, write its children as
\[
    y_\ell=z+u_\ell,
    \qquad
    u_\ell=\rho_zQ_z(q_\ell).
\]
By \eqref{eq:Q-barycenter},
\[
    \sum_{\ell=1}^N u_\ell=0.
\]

We first justify the expansion around $z$.  Since
$p=q+h$ with $|h|\lesssim_M\rho_q$, Lemma~\ref{lem:prev_estimates}
gives
\[
    |h|
    \lesssim_{M,\eta}
    \tau|q-z|.
\]
Thus, for $\tau$ sufficiently small,
\[
    |p-z|\simeq|q-z|.
\]
Similarly,
\[
    |u_\ell|
    \lesssim_M\rho_z
    \lesssim_{M,\eta}
    \tau|q-z|
    \lesssim_{M,\eta}
    \tau|p-z|.
\]
Hence the Taylor expansion in $u_\ell/(p-z)$ is uniform.

The linear term vanishes after summation in $\ell$, and therefore
\[
\sum_{\widehat y=z}(p-y)^{-d}
-
N(p-z)^{-d}
=
O_M\left(
\frac{N\rho_z^2}{|p-z|^{d+2}}
\right).
\]
Consequently,
\begin{align*}
N\rho_p^d\,|E_{n,d}^{\mathrm{cl}}(p)|
&\lesssim_M
N^2\rho_p^d
\sum_{\substack{z\in K_{n-1}\\z\ne q}}
\frac{\rho_z^2}{|p-z|^{d+2}}\\
&\lesssim_{M,\eta}
N^2
\left(\frac{\rho_p}{\rho_q}\right)^d
\sum_{\substack{z\in K_{n-1}\\z\ne q}}
\frac{\rho_q^d\rho_z^2}{|q-z|^{d+2}}.
\end{align*}
Here we used $|p-z|\simeq|q-z|$ in the second line.

Apply Lemma~\ref{lem:prev_estimates} with exponent
\[
    D=d+2
\]
and $a=2$.  Since $D-a=d$, it gives
\[
\sum_{\substack{z\in K_{n-1}\\z\ne q}}
\frac{\rho_q^d\rho_z^2}{|q-z|^{d+2}}
\lesssim_{M,\eta}
\tau^{d+2}.
\]
Therefore
\[
N\rho_p^d\,|E_{n,d}^{\mathrm{cl}}(p)|
\lesssim_{M,\eta}
N^2
\left(\frac{\tau}{N}\right)^d
\tau^{d+2}
=
N^{2-d}\tau^{2d+2}.
\]
Since $d\ge1$ and $0<\tau<1$,
\[
    N^{2-d}\tau^{2d+2}
    \le
    N\tau^{d+2},
\]
which proves \eqref{eq:Ecl-field-new}.  Dividing by $N\tau^d$
gives \eqref{eq:Ecl-field-normalized}.
\end{proof}

Our main result in this section follows by combining the previous Lemmas.

\begin{lemma}[Discrete potential estimates]
\label{lem:higher-potential-new}
For $M\ge2$ and $n\ge1$, assume that
\eqref{eq:old-scale-bound}, \eqref{eq:Q-close-id}, and
\eqref{eq:Q-barycenter} hold for all generations up to $n$.
If $N$ is sufficiently large, then
\begin{equation}
\label{eq:Lambda-main-new}
|\Lambda_{m,d}(p)|
\le
C_{M,\eta}N\tau^d,
\qquad
1\le d\le M,
\end{equation}
for every $1\le m\le n$ and every $p\in K_m$.
\end{lemma}

\begin{proof}
We argue by induction on $m\le n$.  For $m=1$,
Lemma~\ref{lem:field-decomposition} gives $F_{1,d}=S_{1,d}$, so
Lemma~\ref{lem:field-local} yields
\[
    X_1\le C_{M,\eta}.
\]

Now let $2\le m\le n$.  Multiply \eqref{eq:field-recursion-new}, with
$m$ in place of $n$, by $N\rho_p^d$, divide by $N\tau^d$, and use
Lemmas~\ref{lem:field-local}, \ref{lem:field-evaluation}, and
\ref{lem:field-cluster}.  We obtain
\begin{align}
X_{m,d}
\le{}&
C_{M,\eta}
+
C_{M,\eta}\tau X_{m-1,d}
+
C_{M,\eta}
\sum_{t=1}^{M-d}
\tau^{t+1}X_{m-1,d+t}
+
C_{M,\eta}\tau^2
+
C_{M,\eta}\tau^{M+1}.
\label{eq:Xnd-recursion}
\end{align}
When $d=M$, the sum in $t$ is empty.

Since
\[
    X_{m-1,d+t}\le X_{m-1}
\]
and, for $\tau\le1/2$,
\[
\sum_{t=1}^{M-d}\tau^{t+1}
\le
\frac{\tau^2}{1-\tau}
\le
2\tau^2,
\]
we obtain, after enlarging the constant,
\begin{equation}
\label{eq:Xn-recursion}
    X_m
    \le
    C_{M,\eta}
    +
    C_{M,\eta}
    \bigl(\tau+2\tau^2\bigr)X_{m-1}.
\end{equation}

Since $ \tau=N^2r=N^{2-1/s}\rightarrow0$ as $N\rightarrow \infty$ whenever $s<1/2$, we may choose $N$ so large that
\[
    C_{M,\eta}\bigl(\tau+2\tau^2\bigr)
    \le\frac12.
\]
Then
\[
    X_m
    \le
    C_{M,\eta}
    +
    \frac12X_{m-1}.
\]
Together with the bound for $X_1$, iteration gives
\[
    X_m\le C_{M,\eta},
    \qquad 1\le m\le n.
\]
Equivalently,
\[
|\Lambda_{m,d}(p)|
\le
C_{M,\eta}N\tau^d,
\qquad
1\le d\le M,
\]
for every $1\le m\le n$ and every $p\in K_m$.
\end{proof}
%
%
%

\section{The main construction}
\label{sec:finite-renormalization}

We now complete the two iterative procedures separated in the preceding
sections to provide our main construction. The main tool is Lemma \ref{lem:polynomial-renormalization}, which says that as long as $K_0, \dotsc, K_{n-1}$ have been constructed so as to satisfy the required assumptions, then one can choose a family of deformation polynomials $Q_p$ for $p \in K_{n-1}$ so that the next generation, $K_n$, is such that $g_n(x) - g_{n-1}(\hat{x})$ has small oscillation. More explicitly, we apply Lemma \ref{lem:higher-basic-error} to write the oscillation as 
\begin{equation}
	\label{eq:renormalization-input}
	\Delta_1(k)+\Delta_2(k)+\Delta_3(k)
	=
	(N-1)\log\rho_p+b_*+\beta_k
	+
	\sum_{m=1}^{M-1}
	\bigl(
	LP_{m+1,p}(q_k)+R_{m,p}(q_k)
	\bigr)
	+
	E_{M,p}(q_k),
\end{equation}
where
\[
R_{m,p}\in\mathcal P_m,
\qquad
\|R_{m,p}\|_{\mathrm{coeff}}
\lesssim_M N\tau^m,  \; \; \text{ and } |E_{M,p}(q_k)|
\lesssim_M N\tau^M.
\]
The crucial point is the triangular dependence in Lemma \ref{lem:higher-basic-error}, which in turn will allow us to cancel all the polynomial terms $LP_{m+1, p}(q_k) + R_{m,p}(q_k)$ in the above sum, provided that the $P_{m+1,p}$ are well-chosen. This cancellation, plus the smallness on $E_{M,p}(q_k)$, is how we show that $\Delta_1(k) + \Delta_2(k) + \Delta_3(k)$ has small oscillation.

With this main tool, we prove Lemma \ref{lem:finite-order-induction} (the analog of Lemma \ref{lem:basic-ind}), which is our main result. The proof of this Lemma is similar to that of Lemma \ref{lem:basic-ind}, where instead now our estimates for the $\Delta_i(k)$ come from Lemma \ref{lem:polynomial-renormalization}, instead of Lemma \ref{lem:basic-error}. This in turn yields our main result, Theorem \ref{thm:main-general-new}.

\begin{lemma}[Choice of polynomial corrections]
	\label{lem:polynomial-renormalization}
	Fix $M\ge2$, $\eta >0$, and assume that the sets $K_0, \dotsc K_{n-1}$ have been constructed so as to satisfy the assumptions \eqref{eq:old-children}-\eqref{eq:old-scale-bound} and additionally, that 
	\begin{equation}
		\label{eq:field-hypothesis-renormalization}
		|\Lambda_{n-1,d}(p)|
		\lesssim_M
		N\tau^d,
		\qquad
		1\le d\le M,
	\end{equation}
	holds uniformly for $p\in K_{n-1}$.
	
	Then if $N$ is sufficiently large, depending on $M$, $s$, and $\eta$,  there exists a choice of polynomials
	\begin{align}
		P_{m+1,p}, \qquad p \in K_{n-1},  \; \; 1 \le m \le M-1
	\end{align}
	for which \eqref{eq:new-correction-bounds} and
	\eqref{eq:new-correction-normalization} hold, and for which each $x = x_k = p + \rho_p Q_p(q_k) \in K_n$ satisfies
	\begin{equation}
		\label{eq:renormalized-error}
		\Delta_1(k)+\Delta_2(k)+\Delta_3(k)
		=
		(N-1)\log\rho_p+b_*+\beta_k+E_{M,p}(q_k),
	\end{equation}
	where
	\begin{equation}
		\label{eq:renormalized-error-bound}
		|E_{M,p}(q_k)|
		\lesssim_{M, \eta }
		N\tau^M.
	\end{equation}	
	Moreover,
	\begin{equation}
		\label{eq:renormalized-Q-derivative}
		\|Q_p-\mathrm{Id}\|_{C^1([-1,1])}
		\lesssim_M\tau,
	\end{equation}
	and
	\begin{equation}
		\label{eq:renormalized-barycenter}
		\sum_{k=1}^NQ_p(q_k)=0.
	\end{equation}
\end{lemma}

\begin{proof}
	The proof proceeds by induction on $1 \le m \le M-1$, as follows. For $m=1$ (which corresponds to choosing a single quadratic $P_{2,p}$), we observe that Lemma \ref{lem:higher-basic-error} (with each $Q_p \equiv \mathrm{Id})$ gives us that 
	\begin{align*}
		\Delta_1(k) + \Delta_2(k) + \Delta_3(k) = (N-1) \log \rho_p + b_* + \beta_k + R_{1,p}(q_k) + E_{2,p}(q_k)
	\end{align*}
	with $\| R_{1,p} \|_{\mathrm{coeff}} \lesssim_{M} N \tau = N^3 r$, and $\abs{E_{2,p}} \lesssim_{M, \eta } N \tau^2 = N^5 r^2$. The reader should compare this estimate with Lemma \ref{lem:basic-error}; the additional information we have gained, is that the low order error in $N^3 r$ is a polynomial in $q_k$. By Lemma \ref{lem:L_prop}, we may choose a polynomial $P_{2,p}$ of degree at most $2$ for which 
	\begin{align*}
		\sum_{k=1}^N P_{2,p}(q_k) = 0, \; \; \; LP_{2,p} = -R_{1,p},  \; \; \; \text{ and }  \| P_{2,p} \|_{C^1([-1,1])} \lesssim_M N^{-1} \| R_{1,p} \|_{\mathrm{coeff}} \lesssim_M \tau.
	\end{align*}
	Notice that we may guarantee the first cancellation condition above, since $L(\mathrm{const}) = 0$. Once again applying Lemma \ref{lem:higher-basic-error}, but now with $Q_p(t) = t + P_{2,p}(t)$, (which indeed satisfies \eqref{eq:new-correction-bounds} and \eqref{eq:new-correction-normalization}, by the above), we then obtain the new errors
	\begin{align*}
		\Delta_1(k) & + \Delta_2(k) + \Delta_3(k) \\
		& = (N-1) \log \rho_p + b_* + \beta_k + L P_{2,p}(q_k) + R_{1,p}(q_k)   + E_{2,p}(q_k) \\
		& = (N-1) \log \rho_p + b_* +\beta_k  +  E_{2,p}(q_k),
	\end{align*}
	by the choice of $P_{2,p}$, and from the fact that $R_{1,p}$ is a linear polynomial which is independent of $P_{2,p}$. Moreover, that same Lemma yields the bounds $ \abs{E_{2,p}(q_k)} \lesssim_{M, \eta} N \tau^2$. This is exactly the desired base case for $m = 1$. 
	
	Assume now that for some $1 \le m \le M-2$, we have chosen the polynomials $P_{2,p}, \dotsc, P_{m+1, p}$ in such a way that for each $1 \le j \le m$, 
	\begin{align}\label{eq:desired_bds}
		\sum_{k=1}^{N} P_{j+1, p}(q_k) = 0, \; \; \; L P_{j+1,p} = - R_{j,p}, \; \; \; \text{ and } \| P_{j+1,p}  \|_{C^1([-1,1])} \lesssim_M \tau^j.
	\end{align}
	Here, $R_{1,p} \in \mathcal{P}_1$ is as in the base case, and for $j \ge 2$, $R_{j,p} \in \mathcal{P}_{j}$ is the remainder polynomial given by Lemma \ref{lem:higher-basic-error} with the deformation $Q_p(t) = t + P_{2,p}(t) + \cdots +  P_{j, p}(t)$. 
	Since each $R_{j,p}$ depends \textit{just} on the polynomials $P_{\ell, y}$ for $y \in K_{n-1}$ and $1 \le \ell \le j$, we have that $R_{j, p} \in \mathcal{P}_{j}$ is also the same remainder polynomial as in Lemma \ref{lem:higher-basic-error} for the deformation $Q_p(t) = t + P_{2,p}(t) + \cdots + P_{m+1,p} (t) + P_{m+2, p}(t)$ whenever $P_{m+2,p}(t)$ has degree at most $m+2$ and satisfies $\| P_{m+2,p} \|_{C^1([-1,1])} \lesssim_M \tau^{m+1}$. In particular, if we apply Lemma \ref{lem:higher-basic-error} for this deformation, we obtain the error	 
	\begin{align*}
		\Delta_1(k) & + \Delta_2(k) + \Delta_3(k) \\
		& = (N-1) \log \rho_p + b_* + \beta_k + \sum_{j=1}^{m+1} \left( LP_{j+1,p}(q_k) + R_{j,p}(q_k) \right) + E_{m+2, p}(q_k) \\
		& = (N-1) \log \rho_p + b_* + \beta_k + LP_{m+2,p}(q_k) + R_{m+1, p}(q_k)   +  E_{m+2, p}(q_k),
	\end{align*}	
	since the sum over $1 \le j \le m$ cancels exactly. Using Lemma \ref{lem:L_prop}, we then choose for each $p \in K_{n-1}$ some polynomial $P_{m+2, p} \in \mathcal{P}_{m+2}$ satisfying
$$
		\sum_{k=1}^N P_{m+2,p}(q_k) = 0, \qquad LP_{m+2,p} = -R_{m+1,p},$$
and		
		   $$ \| P_{m+2,p} \|_{C^1([-1,1])} \lesssim_M N^{-1} \| R_{m+1,p} \|_{\mathrm{coeff}} \lesssim_M \tau^{m+1},
$$
	where the last inequality is a consequence of Lemma \ref{lem:higher-basic-error}. This proves the induction claim on the choice of the polynomials $P_{j+1, p}$. Moreover, from the above inequality, we deduce also that with the choice of deformation $Q_p(t) = t + P_{2,p}(t) + \cdots + P_{m+2,p}(t)$, the error satisfies
	\begin{align*}
		\Delta_1(k) + \Delta_2(k) + \Delta_3(k) = (N-1) \log \rho_p + b_* + \beta_k + E_{m+2,p}(q_k),
	\end{align*}
	which for $m+1 = M-1$, is the desired conclusion of our Lemma. What is left to check are \eqref{eq:renormalized-Q-derivative} and \eqref{eq:renormalized-barycenter}, but these are clear for $N$ large by our choice of the polynomials $P_{j+1, p}$ and the bounds \eqref{eq:desired_bds}. 
\end{proof}

The preceding lemma shows how to construct the polynomial deformations at a single generation $K_{n}$ given that the past generation $K_{n-1}$ has been constructed. The next Lemma completes the induction on the generations $n \in \N$, giving Lemma \ref{lem:basic-ind} but for the larger range of dimensions $ 0 < s < 1/2$.

\begin{lemma}[Key induction lemma, $s < 1/2$]
\label{lem:finite-order-induction}
Fix $s \in (0,1/2)$, and $\eta>0$ small. Then for $N$ large enough, depending on $s$ and $\eta$, one can construct the sets $K_0, K_1, \dotsc, K_n$ of the form $K_0 = \{0\}$, 
\begin{align*}
	K_{n} = \bigcup_{p \in K_{n-1}} \left \{ x = p + a_{n-1}(p) r^{n-1} Q_p(q_k) \; : \; 1 \le k \le N    \right \} ,
\end{align*}
so that for each $n \ge 1$ and $ x \in K_n$, we have 
\begin{align}\label{eq:general-invariant-new}
	g_n(x) & \coloneqq  N^{-n} \sum_{\substack{ y \in K_n \\ y \ne x   }} \log\abs{x-y} =  c_n + N^{-n} \left( \beta_{j(x)} + e_n(x)  \right), \qquad \abs{e_n(x)} \le \eta,
\end{align}
where $c_n \in \R$ converges as $n \rightarrow \infty$. Moreover, for every $n \ge 2$ and $p \in K_{n-1}$, we have 
\[
	1 \le a_{n-1}(p) \lesssim_\eta N \sin \theta_{j(p)}
\]
and additionally, the $Q_p$ satisfy 
\begin{align*}
	\| Q_p - \mathrm{Id}   \|_{ C^{1}([-1,1]) } \lesssim_{s, \eta} N^2 r \ll 1,  \; \;\; \;  \sum_{k=1}^N Q_p(q_k) = 0.  
\end{align*}
\end{lemma}

\begin{proof}
Fix $s \in (0,1/2)$ and $\eta >0$ small. Recall that for $N \ge 1$, we set $N r^s = 1$ and $\tau = N^2 r$. Since $s < 1/2$, notice that we may choose some fixed $M \ge 2$ large enough so that 
\begin{equation}
	\label{eq:finite-order-smallness}
	C_MN\tau^M  = C_M N^{2M+1} r^M = C_M N^{1 + M (2 -1/s)} \le \eta,
\end{equation} 
for all $N$ large enough depending on $M$ and $\eta$. It suffices to choose, for example, $M > s/(1-2s)$. In what follows, we will choose $N$ large enough, depending on $s$ and $\eta$ without explicitly stating, so as to apply our previous Lemmas where needed.

At generation one, choosing $a_0(0) = 1$, we may write for $x = q_k \in K_1$,
\begin{align*}
	g_1(x) = g_1(q_k) = N^{-1} \sum_{j \ne k} \log \abs{q_k -q_j} = \frac{b_k}N = \frac{b^*}N + \frac{\beta_k}N
\end{align*}
With the constant $c_1 = b^*/N$, we see that $e_1(x) = e_1(q_k) = 0$.
Of course the first choice, $a_0(0) =1$ satisfies $1 \le a_0(0) \lesssim_\eta N$.

We now argue by induction on the generation $n$. Namely, assume we have thus far constructed $K_{n-1}$ so that for all $1 \le m \le n-2$ and all $z \in K_m$, we have
\begin{align}\label{eqn:ind_hypothesis}
	1 \le a_{m}(z) \lesssim_\eta N \sin \theta_{j(z)}, \; \; \sum_{k=1}^N Q_z(q_k) = 0, \; \; \|Q_z - \mathrm{Id} \|_{C^1([-1,1])}  \lesssim_M \tau,
\end{align}
and so that for each $p \in K_{n-1}$
\begin{equation}
	\label{eq:induction-parent-simple}
	g_{n-1}(p)
	=
	c_{n-1}
	+
	N^{-(n-1)}
	\bigl(
	\beta_{j(p)}+e_{n-1}(p)
	\bigr),
	\qquad
	|e_{n-1}(p)|\le\eta.
\end{equation}
We now aim to build $K_n$, but we need to do so in a non-circular way: the first step is to choose the $a_{n-1}(p)$ for $p \in K_{n-1}$. Just as in the construction of Section \ref{sec:basic-construction}, we define $a_{n-1}(p)$ by the equation 
\begin{equation}
	\label{eq:scalar-choice-simple}
	\log a_{n-1}(p)
	=
	\frac{N}{N-1}
	\bigl(
	\eta-\beta_{j}-e_{n-1}(p)
	\bigr),
\end{equation}
where we write $j = j(p)$ for the index of $p$. Since $\beta_j\le0$ from Lemma \ref{lem:cheb} and since $\abs{e_{n-1}(p)} \le \eta$ by induction, 
we have $a_{n-1}(p)\ge1$. Moreover, using $\beta_j
=
\log\tfrac{\sin\theta_1}{\sin\theta_j}$ and $\sin\theta_1\simeq N^{-1}$, we estimate from above,
\begin{align*}
	a_{n-1}(p)
	&=
	\exp\left[
	\frac{N}{N-1}
	\bigl(
	\eta-\beta_j-e_{n-1}(p)
	\bigr)
	\right]
	\\
	&\lesssim_\eta
	\exp\left(
	-\frac{N}{N-1}\beta_j
	\right) =
	\left(
	\frac{\sin\theta_j}{\sin\theta_1}
	\right)^{N/(N-1)}
	\lesssim_\eta
	N\sin\theta_j.
\end{align*}
Thus with this choice of $a_{n-1}(p)$, we indeed have the bounds 
\begin{equation}
	\label{eq:geometric-bound-simple}
	1
	\le
	a_{n-1}(p)
	\lesssim_\eta
	N\sin\theta_j,
\end{equation}
and so all of our scalars $a_{n-1}$ satisfy \eqref{eq:current-scale-bound}, the induction hypothesis for the next generation. 

Our basic estimates by induction \eqref{eqn:ind_hypothesis} show that the estimates from Section \ref{sec:higher-potential} hold. Indeed, we may apply Lemma \ref{lem:higher-potential-new} to the generation $n-1$ to see that 
\begin{equation}
	\label{eq:field-at-current-generation}
	|\Lambda_{n-1,d}(p)|
	\lesssim_M
	N\tau^d,
	\qquad
	1\le d\le M,
	\qquad
	p\in K_{n-1}.
\end{equation}
Using the above bounds, we see that we can apply Lemma~\ref{lem:polynomial-renormalization} to obtain the polynomials
\[
P_{2,p},\ldots,P_{M,p},
\qquad
p\in K_{n-1},
\]
satisfying \eqref{eq:new-correction-bounds} and
\eqref{eq:new-correction-normalization}, and define the new deformations $Q_{p}$  for $p \in K_{n-1}$ as in \eqref{eq:new-deformation}: 
\begin{align*}
	Q_p(t) = t + P_{2,p}(t) + \cdots + P_{M,p}(t), \; \; p \in K_{n-1}.
\end{align*}
Then $K_{n}$ is constructed from $K_{n-1}$ as always, by replacing $p \in K_{n-1}$ by its $N$ children $x_k$ as in \eqref{eq:new-children}, with $\rho_p=a_{n-1}(p)r^{n-1}$:
\begin{align*}
	x_k = p + \rho_p Q_p(q_k), \; \; 1 \le k \le N.
\end{align*}
Notice that Lemma \ref{lem:polynomial-renormalization} also gives us our desired estimates on the maps $Q_p$ for $p \in K_{n-1}$:
\begin{align*}
	\sum_{k=1}^N Q_p(q_k) = 0, \; \; \| Q_p - \mathrm{Id}  \|_{C^1([-1,1])} \lesssim_M \tau.   
\end{align*}
Having completed the construction of \(K_n\), we turn to the estimates
 estimates on $g_n (x)$. 

 Let $x = x_k \in  K_n$ have index $k$ and parent $\hat{x} = p \in K_{n-1}$. We again write $j = j(p)$ for the index of $p$. As in the argument preceding Lemma \ref{lem:basic-ind}, we write 
\begin{align*}
	g_{n}(x_k) - g_{n-1}(p) = N^{-n} \left( \Delta_1(k) + \Delta_2(k) + \Delta_3(k)  \right),
\end{align*}
and Lemma \ref{lem:polynomial-renormalization} tells us that our $Q_p$ are chosen so as to yield
\begin{equation}
	\label{eq:renormalized-step-simple}
	\Delta_1(k)+\Delta_2(k)+\Delta_3(k)
	=
	(N-1)\log\rho_p+b_*+\beta_k+E_{M,p}(q_k),
\end{equation}
with error satisfying 
\begin{equation}
	\label{eq:EM-simple}
	|E_{M,p}(q_k)|
	\lesssim_M
	N\tau^M.
\end{equation}
In particular, we may write using the induction hypothesis,
\begin{align*}
	g_n(x_k)
	&=
	g_{n-1}(p)
	+
	N^{-n}
	\bigl(
	\Delta_1(k)+\Delta_2(k)+\Delta_3(k)
	\bigr)
	\\
	&=
	c_{n-1}
	+
	N^{-n+1}
	\bigl(
	\beta_j+e_{n-1}(p)
	\bigr)
	\\
	&\quad
	+
	N^{-n}
	\Bigl[
	(N-1)
	\bigl(
	\log a_{n-1}(p)+(n-1)\log r
	\bigr)
	+
	b_*+\beta_k+E_{M,p}(q_k)
	\Bigr].
\end{align*}
By \eqref{eq:scalar-choice-simple}, the choice of $a_{n-1}(p)$, we have 
\[
N\bigl(
\beta_j+e_{n-1}(p)
\bigr)
+
(N-1)\log a_{n-1}(p)
=
N\eta,
\]
so we can rewrite 
\[
g_n(x_k)
=
c_n
+
N^{-n}
\bigl(
\beta_k+ e_n(x)
\bigr),
\]
where
\begin{equation}
	\label{eq:cn-recursion}
	c_n
	=
	c_{n-1}
	+
	N^{-n}
	\bigl[
	(N-1)(n-1)\log r+b_*+N\eta
	\bigr], \; \; \; e_{n}(x) = E_{M,p}(q_k) 
\end{equation}
Since $\abs{e_n(x)} =  \abs{E_{M, p}(q_k)} \le C_M N \tau^M \le \eta$ for $N$ large enough, by \eqref{eq:finite-order-smallness}, this concludes the proof by induction. Moreover, by \eqref{eq:cn-recursion},
\[
|c_n-c_{n-1}|
\le C_{N,r,\eta}\,nN^{-n}.
\]
Since this bound is summable in $n$, the sequence
$(c_n)$ converges as $n\to\infty$. 
\end{proof}

Just as in Theorem \ref{thm:one-third}, one uses Lemma \ref{lem:finite-order-induction} to deduce our main result:
\begin{theorem}\label{thm:main-general-new}
	For every $0 < s< 1/2$, there exists an $s$-Ahlfors regular compact set $K \subset \R \subset \C$ so that the harmonic measure in $\C \setminus K$ with pole at infinity $\omega^\infty$ is equivalent to $\mathcal{H}^s|_K$, i.e., there is some $C \ge 1$ so that 
	\begin{align*}
		 C^{-1} \, \mathcal{H}^s|_K(F)\le \omega^\infty(F) \le C \mathcal{H}^s|_K(F) \text{ for all } F \subset K \text{ Borel.}
	\end{align*}
	In particular, $\dim \omega^\infty = s$. 
\end{theorem}

\bibliographystyle{alpha}
\bibliography{bibl}

\end{document}